\documentclass[12pt]{amsart}
\usepackage[utf8]{inputenc}
\usepackage{graphicx}
\usepackage{graphics}
\usepackage{stmaryrd}
\SetSymbolFont{stmry}{bold}{U}{stmry}{m}{n}
\usepackage{amsmath,amssymb,mathrsfs}
\usepackage{a4wide}
\usepackage[T1]{fontenc}
\usepackage{verbatim}
\usepackage{caption}
\usepackage{moreverb}
\usepackage{color}
\usepackage{float}
\usepackage{caption}
\usepackage{subcaption}
\usepackage{nccmath}
\usepackage[hidelinks]{hyperref}
\usepackage{amsthm}
\usepackage{enumerate}
\usepackage{thmtools}
\usepackage[shortlabels]{enumitem}
\usepackage{centernot}
\usepackage{bm}
\usepackage{physics}
\usepackage{mathtools}
\usepackage{cleveref}
\usepackage{dsfont}
\usepackage{titlesec}
\usepackage[foot]{amsaddr}

\theoremstyle{plain}
\newtheorem{thm}{\textbf{Theorem}}

\theoremstyle{definition}

\theoremstyle{remark}
\newtheorem*{rem}{\textbf{\textup{Remark}}}

\theoremstyle{plain}
\newtheorem{prop}{\textbf{Proposition}}

\theoremstyle{plain}
\newtheorem{lem}{\textbf{Lemma}}

\makeatletter
\renewenvironment{proof}[1][\textbf {Proof.}]{\par
\pushQED{$\blacksquare$}%
\normalfont \topsep6\p@\@plus6\p@\relax
\trivlist
\item\relax
{
#1\@addpunct{.}}\hspace\labelsep\ignorespaces
}{%
\popQED\endtrivlist\@endpefalse
}

\newenvironment{qedequation*}{%
   \pushQED{\qed}%
   \mathdisplay@push
   \st@rredtrue \global\@eqnswfalse
   \mathdisplay{equation*}%
}{%
   \endmathdisplay{equation*}%
   \mathdisplay@pop
   \ignorespacesafterend
   \popQED\@endpefalse
}

\makeatother

\usepackage{xpatch}
\makeatletter
\AtBeginDocument{\xpatchcmd{\@thm}{\thm@headpunct{.}}{\thm@headpunct{}}{}{}}
\makeatother

\newcommand*{\defeq}{\mathrel{\vcenter{\baselineskip0.5ex \lineskiplimit0pt
                     \hbox{\scriptsize.}\hbox{\scriptsize.}}}%
                     =}
\newcommand*{\eqdef}{= \mathrel{\vcenter{\baselineskip0.5ex \lineskiplimit0pt
                     \hbox{\scriptsize.}\hbox{\scriptsize.}}}
                     }
\makeatletter
\newcommand{\pushright}[1]{\ifmeasuring@#1\else\omit\hfill$\displaystyle#1$\fi\ignorespaces}
\newcommand{\pushleft}[1]{\ifmeasuring@#1\else\omit$\displaystyle#1$\hfill\fi\ignorespaces}

\newcommand{\RN}[1]{%
  \textup{\uppercase\expandafter{\romannumeral#1}}%
}
\newcommand{\specialcell}[1]{\ifmeasuring@#1\else\omit$\displaystyle#1$\ignorespaces\fi}
\makeatother

\title[Global weak solutions for the LLG-VM system]{Global weak solutions for the Landau-Lifshitz-Gilbert-Vlasov-Maxwell system coupled via emergent electromagnetic fields}
\author[Tvrtko Dorešić]{Tvrtko Dorešić$^*$}
\address{$^*$Department of Mathematics, RWTH Aachen University.}
\author[Christof Melcher]{Christof Melcher$^\dagger$}
\address{$^{\dagger}$Department of Mathematics \& JARA FIT, RWTH Aachen University.}
\email{$^*$doresic@eddy.rwth-aachen.de, $^\dagger$melcher@rwth-aachen.de}

  \date{\today}
  \keywords{Landau-Lifshitz-Gilbert equation, Vlasov-Maxwell system, global weak solutions, micromagnetics, topological solitons}
  \subjclass[2010]{35Q20, 35Q60, 35K55, 35D30}

\begin{document}
\titleformat{\section}[hang]{}{\Large\thetitle.}{0.8em}{\Large\bfseries\centering}
\titleformat{\subsection}[hang]{}{\large\thetitle.}{0.5em}{\large\bfseries}
\titlespacing*{\subsection}{0pt}{*2}{*1.5}
\titleformat{\subsubsection}[hang]{}{\normalsize\thetitle.}{0.5em}{\normalsize\bfseries}
\titleformat{\paragraph}[runin]{}{\large\thetitle.}{0.5em}{\normalsize\bfseries}
\begin{abstract}
Motivated by recent models of current driven magnetization dynamics, we examine the coupling of the Landau-Lifshitz-Gilbert equation and classical electron transport governed by the Vlasov-Maxwell system. The interaction is based on space-time gyro-coupling in the form of emergent electromagnetic fields of quantized helicity that add up to the conventional Maxwell fields. We construct global weak solutions of the coupled system in the framework of frustrated magnets with competing first and second order gradient interactions known to host topological solitons such as magnetic skyrmions and hopfions.
\end{abstract}

\maketitle

\section{Introduction}

We are concerned with a mathematically novel and somewhat unusual combination of PDEs aiming to describe the dynamic interplay of magnetization structures and (collision-free) electric currents, which is essential for various spintronic applications. Magnetization structures are given in terms of an 
$\mathbb S^2$ valued field $\bm m=\bm m(\bm x,t)$ which is governed by a
micromagnetic interaction energy $E=E(\bm m)$, a quadratic integral functional 
that we shall specify later. We focus on a dynamic model, 
where $\bm m$ evolves according to the following Landau-Lifshitz-Gilbert equation (LLG)
\begin{equation}\label{eq:abstract_LLG}
    \partial_t \bm m +  (\bm j \cdot \nabla) \bm m 
    = \bm m \times \left( \alpha \, \partial_t \bm m  - \bm h_{\rm eff} \right) \, .
\end{equation}
Here $\alpha>0$ is the Gilbert damping factor and $\bm h_{\rm eff}=-\frac{\delta E}{\delta \bm m}(\bm m)$ is the effective field induced by $\bm m$. The dynamics is
driven by a current $\bm j= \bm j(\bm x,t)$ of electrons whose spins are assumed to adiabatically align with the local magnetization direction $\bm m$, giving rise to a spin current $\bm Q = \bm j \otimes \bm m$.  The divergence of $\bm Q$ perpendicular to $\bm m$, featured in \eqref{eq:abstract_LLG}, is called adiabatic spin-transfer torque \cite{ZhangLi2004}.

In the absence of $\bm j$, LLG is a hybrid of heat and Hamiltonian flow of $E$. The spin-transfer torque has the form of a transport term. In a simplified approach it is assumed that $\bm j$ is constant. Due to the hybrid structure, however, the term cannot be eliminated by means of a simple Galilean transformation. In models of current driven domain walls it is customary to include non-adiabatic spin-transfer terms $\beta \bm m \times (\bm j \cdot \nabla) \bm m$ with an additional parameter $\beta$. Existence and well-posedness results for this so-called Landau-Lifshitz-Slonczewski equation have been derived in \cite{MelcherPtashnyk2013, DoeringMelcher2017}. 

In the context of topological phases on very small scales the interplay of electron currents and magnetization structures becomes more complex and multifaceted, calling for a more precise description that takes into account mutual interactions, i.e. the counter-effect of magnetization structures 
on the electron flow. Electron transport is described in term of an electron distribution function $f=f(t,\bm x,\bm v)$ depending on time $t$, position $\bm x\in \mathbb R^3$ and velocity $\bm v \in \mathbb R^3$ so that the current is obtained as the first velocity moment 
$\bm j =q \int_{\mathbb R^3} \bm v f \, d\bm v$. Ignoring collisions, distribution functions generally evolve according to Vlasov equations
\begin{equation*}
    \partial_t f + \bm v \cdot \nabla_{\bm x} f + \bm F \cdot \nabla_{\bm v} f =0 \, ,
\end{equation*}
where $\bm F=q (\bm E + \bm v \times \bm B)$ with $q=-1$ is the Lorentz force induced by electromagnetic fields $\bm E$ and $\bm B$ that satisfy Maxwell's equations in a self-consistent way, giving rise to the Vlasov-Maxwell system, see \cite{Schmeiser1990,Juengel2009} for a detailed discussion in the context of semiconductors.\\
Short-time existence and uniqueness of classical solutions to the Vlasov-Maxwell system have been proven by Wollman in \cite{Wollman1984}
based on a generalization of a general local existence result for quasilinear symmetric hyperbolic systems \cite{Kato1975}.
Global existence of a weak solution with has been obtained by DiPerna and Lions in \cite{DiPernaLions1989} based on a regularization procedure, velocity averaging and the method of renormalization. \\

Coupling to the magnetization structure is based on a recent physical observation in connection with Hall effects in the presence of nontrivial topologies, namely the emergence of virtual electromagnetic fields $\bm e$ and $\bm b$, which contribute to the conventional Maxwell fields on the level of electron transport with a modified Lorentz force 
\begin{equation} \label{eq:total_Lorentz}
\bm F=q \left(\bm E + \bm e + \bm v \times (\bm B + \bm b)\right),
\end{equation}
\cite{Schulz2012, Nagaosa2013}. The emergent fields are derived from the evolving magnetization field $\bm m$, see below. 
The resulting system will be called Landau-Lifshitz-Gilbert-Vlasov-Maxwell system (LLG-VM). A key property of this system is the following energy-dissipation law
\begin{equation} \label{eq:abstract_energy_law}
     \alpha \int_0^T  \|\partial_t \bm m\|^2_{L^2} \, dt + \Big[ \mathbb{E}(f,\bm E, \bm B, \bm m)\Big]_{t=0}^T \le 0 \, ,
\end{equation}
for the total energy
\begin{equation} \label{eq:abstract_energy}
\mathbb{E}(f,\bm E, \bm B, \bm m) = 
      \int_{\mathbb R^3} \int_{\mathbb R^3}  \frac{|\bm v|}{2}^2 f \, d\bm v \, d\bm x +    \frac{1}{2}\left( \varepsilon_r \, \|\bm E\|^2_{L^2} + \frac{1}{\mu_r} \, \|\bm B\|^2_{L^2} \right) + E(\bm m) \, ,
\end{equation}
valid for sufficiently regular solutions and weak limits. Here $\varepsilon_r \geq 1$ and $\mu_r \geq 1$ represent relative permittivity and permeability constants, respectively. Notably, there in no explicit dependence on the emergent fields $\bm e$ and $\bm b$, which indicates that this coupling is indeed natural. Depending on the choice of $E=E(\bm m)$, the resulting a priori bounds are the basis of the our global existence result
for (weak) solutions
\[
(f,\bm E, \bm B, \bm m) \in L^\infty((0,\infty);\{\mathbb{E}< \infty\})
\quad \text{and} \quad \frac{\partial \bm m}{\partial t} \in L^2((0, \infty) \times \mathbb R^3) \, ,
\]
under further requirements on the initial data and with further specifications on the regularity.
To this end we shall focus on a small scale  model for a frustrated magnet that takes into account second order gradient terms\begin{align}\label{eq:frustratedmagnet}
E(\bm m) = \frac{1}{2} \int_{\mathbb{R}^3}  |\nabla^2 \bm m|^2  -  |\nabla \bm m|^2 +  h |\bm m - \bm{\hat {e}}_3|^2 d \bm x.
\end{align}
The first and second order gradient terms account for competing nearest-neighbor ferromagnetic and higher-neighbor antiferromagnetic interaction, respectively. The last term accounts for the interaction with a Zeeman field pointing in the $\bm{\hat {e}}_3$ direction.
In the ferromagnetic regime with $h>1/4$
so that $E(\bm m) \gtrsim \|\bm m - \bm{\hat{e}}_3\|_{H^2}^2$,
such models are known to host topological solitons in space dimension $d=2,3$, magnetic skyrmions and hopfions, respectively, see e.g. \cite{Sutcliffe2017}. A key analytical consequence of a second order gradient term is that LLG behaves subcritical with respect to the energy $E(\bm m)$, i.e., concentration effects are ruled out by the fundamental energy law \eqref{eq:abstract_energy_law}. Moreover, the topology of the field $\bm m$ is preserved under a flow that exhibits the space-time bounds provided by  \eqref{eq:abstract_energy_law}. Finally, the resulting regularity of the emergent fields $\bm e$ and $\bm b$ facilitate the compactness arguments for the transport equation.

Related systems arising in connection with domain wall motion and magnetic switching in multi-layers has been developed amd examined in \cite{Chen2015, Chen2016, Chai2018}. Starting from Schr\"odinger-Poisson equations for spinors, the semiclassical mean-field limit yields a Vlasov-Poission equation for the associated Wigner function coupled to the Landau-Lifshitz-Gilbert.
The coupling is realized by means of a spin transfer term in the effective field, which is induced by Pauli projections, and a source term in the Wigner equation, respectively, rather than an adiabatic spin-transfer torque and emergent electromagnetic fields as in our case. Conventional magnetostatic stray-fields playing a particular role for conventional micromagnetic structures such as domain walls are taken into account as lower order perturbation on the level of LLG. \\
Here our focus is on models to describe the transport magnetic topological solitons occurring on very small scales. Stray-fields are less relevant in this regime and are neglected in our discussion, focussing on the new difficulties due to the lack of regularizing properties of the full Maxwell equations in combination with emergent electromagnetism and topology that we now discuss in more detail.

\subsection*{Emergent electromagnetism and topology}\label{subsection:emergentfields}
A smoothly evolving magnetization field $\bm m$
induces a space-time vorticity, i.e., a two-form $\omega= \frac 1 2 \, \omega_{\mu \nu} \, dx_\mu \wedge dx_\nu$ with components
\begin{align*} 
\omega_{\mu \nu}= \bm m \cdot \left( \partial_\mu \bm m \times \partial_\nu \bm m \right)\, , 
\end{align*}
for space-time indices $0 \le \mu, \nu \le 3$ so that $\partial_0 = \partial_t$. The two-form $\omega$ is the pull-back $\bm m^\ast \omega_{\mathbb{S}^2}$
of the standard volume form $\omega_{\mathbb{S}^2}$ on $\mathbb{S}^2$ by $\bm m$, see e.g. \cite{KurzkeMelcherMoser2011, KurzkeMelcherMoserSpirn2010}.  
It follows that $\omega$ is closed so that Bianchi's identity 
holds true for $0 \le \mu, \nu, \sigma  \le 3$:
\begin{align*}
\partial_\sigma  \omega_{\mu \nu} +\partial_\nu  \omega_{\sigma \mu } +\partial_\mu  \omega_{\nu \sigma} =0 \, .
\end{align*}
In the spirit of the Faraday form from electromagnetism, the decomposition
\begin{align*}
\omega =e_i  \, dx_i \wedge dt +  \frac 1 2 \omega_{jk} \, dx_j \wedge dx_k,
\end{align*}
for spatial indices $1 \le i,j,k \le 3$, gives rise to the previously mentioned
emergent electromagnetic fields $\bm e$ and $\bm b$ with components
\begin{align} \label{eq:emergent_fields}
b_i = \frac{1}{2}\epsilon_{ijk} \, \bm m \cdot  (\partial_j \bm m \times \partial_k \bm m)  \quad \text{and} \quad e_i = \bm m \cdot  (\partial_i \bm m \times \partial_t \bm m).
\end{align}
According to the Bianchi identity the emergent fields satisfy the homogeneous Maxwell equations (Gau{\ss} law for magnetism and Faraday's law)
\begin{align*}
\nabla \cdot \bm b =0 \quad \text{and} \quad \frac{\partial \bm b}{\partial t} + \nabla \times \bm e =0. 
\end{align*}
The emergent electromagnetism bears physical relevance as it captures the interplay between the magnetization structure and electric currents.
Charge $q$ particles traversing at velocity $\bm v$ and momentum $\bm p$ experience an additional Lorenz force $q \left( \bm e + \bm v \times \bm b\right)$, giving rise to the total force $\bm F$ in \eqref{eq:total_Lorentz} such that $\frac{d \bm p}{dt}=\bm F$. There is no universal counter part to the inhomogeneous Maxwell equations on the level of emergent fields.

\medskip

An intriguing new feature compared to conventional electromagnetism is quantization, which gives certain localized magnetization structures the character of charged particles. Such topological solitons in magnetism are classified according to their dimension as skyrmions and hopfions, respectively.
The Gau{\ss} law $\nabla \cdot \bm b=0$ gives rise to a vector potential $\bm a$ such that $\bm b = \nabla \times \bm a$. The scalar product
$\bm a \cdot \bm b$ is the emergent magnetic helicity density. Under suitable decay conditions $\bm m \to \bm{\hat e}_3$ as $|\bm x| \to \infty$, the helicity integral exists and is quantized 
\begin{align*} 
 \int_{\mathbb R^3} \bm a \cdot \bm b \; d \bm x =(4 \pi)^2 \, H(\bm m) \, ,
\end{align*}
where $H(\bm m) \in \mathbb{Z}$ is the Hopf invariant associated to $\bm m$,
considered as a continuous map from the compactification $\mathbb{S}^3$.
The Hopf invariant is a homotopy invariant and describes the topology of the field $\bm m$ in terms of the linking number of two generic fibers of $\bm m$. 
Moreover, the flux $\bm b$ through a hyperplane, say $\mathbb{R}^2$ is also quantized
\begin{align*} 
 \int_{\mathbb R^2} b_3 \; dx =4 \pi \, Q(\bm m) \, , 
\end{align*}
where $Q(\bm m)\in \mathbb Z$ is the Brouwer degree or skyrmion number associated to $\bm m =\bm m|_{\mathbb R^2}$, considered as a continuous map from the compactification $\mathbb{S}^2$. The invariants $Q$ and $H$ are used for topological classification of localized structures in magnetism.

\medskip


 The mathematical theory of topological solitons in magnetism has mainly been developed in the context of two-dimensional structures. The so-called chiral skyrmions have been predicted to exist 20 years before \cite{Bogdanov1989} and owe their stability to the spin-orbit effects (Dzyaloshinskii-Moriya interaction), see e.g. \cite{Melcher2014, LiMelcher2017} for a mathematical account. An alternative stabilization mechanism is based on magnetic frustration, i.e., alternating ferromagnetic and anti-ferromagnetic interaction in a Heisenberg lattice, see e.g.
 \cite{Meyer2019}. A continuum theory for frustrated magnets 
 including \eqref{eq:frustratedmagnet}
 is derived in \cite{Lin2016} and is shown to support two-dimensional skyrmions \cite{Lin2016} as well as their three-dimensional topological counterparts (hopfions) \cite{Sutcliffe2017, Rybakov2019}. Here we focus on this three-dimensional case. Hopfion dynamics have been recently studied in chiral magnets \cite{Wang2019} and frustrated magnets \cite{Liu2020}. The model considered in this work is the combined system of equations of motion that has been suggested in \cite{Nagaosa2013} for the case of magnetic skyrmions.
 
\subsection*{Main result}\label{subsection:LLGVM}
The complete Landau-Lifshitz-Gilbert-Vlasov-Maxwell (LLG-VM) system for the magnetization field $\bm m=\bm m (\bm x, t)$,
the distribution function $f=f(\bm x, \bm v, t)$ and the 
electromagnetic fields
$\bm E=\bm E(\bm x, t)$ and $\bm B=\bm B(\bm x, t)$ with
$(\bm x, \bm v, t) \in \mathbb{R}^3 \times \mathbb{R}^3 \times (0, \infty)$ has the following explicit form:
The Landau-Lifshitz-Gilbert equation
\begin{equation}\label{eq:LLG}
    \partial_t \bm m = \bm m \times ( \alpha\, \partial_t \bm m + \Delta \bm m + \Delta^2 \bm m-h\, \bm{\hat{e}}_3) - (\bm j \cdot \nabla) \bm m,  
\end{equation}
inducing emergent fields $\bm e$ and $\bm b$ given by
\eqref{eq:emergent_fields},
is coupled to the Vlasov-Maxwell system
\begin{equation} \label{eq:Vlasov-1}
\partial_t f + \bm v \cdot \nabla_x f - (\bm E + \bm e + \bm v \times(\bm B + \bm b)) \cdot \nabla_v f =0 \, ,
\end{equation}
with the homogeneous Maxwell equations
\begin{equation}
\partial_t \bm B + \nabla \times \bm E = 0 
\quad \text{and} \quad \nabla \cdot \bm B = 0 \label{eq:Maxwell4} \, ,
\end{equation}
and the inhomogeneous Maxwell equations
\begin{equation}
 \partial_t \bm D -  \, \nabla \times \bm H = - \bm j
 \quad \text{and} \quad \nabla \cdot \bm D = \rho \label{eq:Maxwell-1} \, ,
\end{equation}
with constitutive laws in terms of relative permittivity and permeability $\varepsilon_r \geq 1$ and $\mu_r \geq 1$ 
\begin{equation}
\bm D = \varepsilon_r \bm E \quad \text{and}
\quad \bm B = \mu_r \bm H \, ,
\end{equation}
and current and charge densities
\begin{equation*}
    \bm j =- \int_{\mathbb R^3} \bm v f \, d\bm v \quad \text{and} \quad \rho = -\int_{\mathbb R^3} f \, d\bm v. \end{equation*}

\subsubsection*{Notation and function spaces}
By $\mathscr{D}(\mathbb R^3)$ we denote the space of infinitely smooth functions with compact support, and by $\mathscr{D}'(\mathbb R^3)$ the 
space of Schwartz distributions.
For any $s\in\mathbb R$ and $1<p<\infty$, with $W^{s,p}(\mathbb R^N)$ we denote the fractional Sobolev space
\begin{equation*}
    W^{s,p}(\mathbb R^N) \defeq \big( I-\Delta \big)^{-s/2}L^p(\mathbb R^N) \, .
\end{equation*}
 In the case $p=2$ we let $H^s(\mathbb R^N)\defeq W^{s,2}(\mathbb R^N).$ We refer to \cite{Taylor3} for the definition on domains and other properties. In the formulation of the Theorems we use the following notation
\begin{equation*}
    H^s(\mathbb R^3 ;\mathbb S^2) \defeq \{ \bm m :\mathbb R^3 \rightarrow \mathbb S^2 : \bm m - \bm{\hat {e}}_3 \in H^s(\mathbb R^3;\mathbb R^3) \} \, .
\end{equation*}
With $C>0$ we always denote a generic constant that may change from line to line.
With $\langle \cdot,\cdot \rangle$ we denote the $L^2$ scalar product with respect to the spatial varible $\bm x$.
With $B_R$ we denote an open ball in $\mathbb R^3$ with radius $R>0$.
With $\mathcal{F}$ we denote the Fourier transform.

\subsubsection*{Global weak solutions of the LLG-VM system}

We are concerned with global existence of distributional solutions for initial data in the energy space. Under some further integrability properties on the initial distribution
function we obtain the following existence result:
\begin{thm}\label{thm:glavni}
Let $\bm m_0 \in H^2 (\mathbb R^3;\mathbb S^2)$ and $h>1/4\,$. Let $f_0\in L^1\cap L^r(\mathbb R^3 \times \mathbb R^3)$ for $r>3$ be non-negative and satisfy
\begin{equation*}
    \int\int_{\mathbb R^3 \times \mathbb R^3} f_0 |\bm v|^2 \, d\bm x \, d\bm v < \infty.
\end{equation*}
Let $\bm E_0,\bm B_0 \, \in L^2(\mathbb R^3; \mathbb R^3)$ satisfy the following compatibility condition:
\begin{equation*}
    \nabla \cdot \bm E_0 = - \int_{\mathbb R^3} f_0 \, d\bm v\, , \quad \nabla \cdot \bm B_0 = 0 \quad \text{in} \quad \mathscr{D}'(\mathbb R^3) \, .
\end{equation*}
Then there exist
\begin{gather*}
\bm m \in L^\infty((0,\infty);H^2 (\mathbb R^3;\mathbb S^2)) \cap \dot H^1((0,\infty);L^2(\mathbb R^3;\mathbb R^3)) \, , \\
f\in L^\infty((0,\infty) \, ;L^1 \cap L^r ( \mathbb R^3 \times \mathbb R^3))\, , \\
\bm E, \bm B \in L^\infty((0,\infty) ; L^2(\mathbb R^3;\mathbb R^3)) \, , 
\end{gather*}
which satisfy the LLG-VM system \eqref{eq:LLG}-\eqref{eq:Maxwell4} in the sense of distributions such that
\begin{alignat*}{3}
    &\hphantom{\bm E,}\bm m \in C([0,\infty)\, ;\, H^s(\mathbb R^3 \,;\mathbb S^2)) \quad &&\text{for all } s<2 \, , \\
   &f \in C([0,\infty)\, ; \,W_{loc}^{-s,\,p}(\mathbb R^3 \times \mathbb R^3)) \quad &&\text{for all } s>0 \, , \; p\in [3r/(2r+3),\, 3] \, , \\
   &\bm E,\, \bm B \in C([0,\infty)\, ;\, H^{-s}_{loc}(\mathbb R^3\,;\mathbb R^3)) \quad &&\text{for all } s>0 \, ,
\end{alignat*}
where $f$ is non-negative and $f|_{t=0} = f_0,\ \bm E|_{t=0}=\bm E_0, \ \bm B|_{t=0}=\bm B_0,\ \bm m|_{t=0} = \bm m_0$.
\end{thm}

\paragraph{Remarks.}

\begin{enumerate}
\item The magnetization field $\bm m$ obtained in Theorem \ref{thm:glavni} is, by virtue of Sobolev embedding, continuous in space-time and therefore topology preserving. Alternatively, it can be shown that 
$t \mapsto H(\bm m(t)) \in \mathbb{Z}$ is continuous and therefore constant, see Lemma \ref{lem:hopfion}. The model is therefore suitable to describe the current driven dynamics skyrmions and hopfions in frustrated magnets.

\item The solution $(\bm m, f, \bm E, \bm B)$ satisfies the energy inequality \eqref{eq:abstract_energy_law}. An interesting open question concerns strong continuity in the energy norm, i.e., whether \eqref{eq:abstract_energy_law} upgrades to an identity, which is unknown even for weak solutions of the Vlasov-Maxwell system alone.

\item Due to the strong regularizing effect of the governing micromagnetic energy \eqref{eq:frustratedmagnet}, the free LLG equation \eqref{eq:LLG} for $\bm j=0$ allows for regular solutions $\bm m$ up to any order.
In the coupled system, however, the regularity of $\bm m$ is limited by the regularity of $\bm j$ arising from $f$.

\item The critical regularity assumption for $f_0$ is $r=3$ in order to have integrability of the coupling term $\bm e \, f$ in the Vlasov equation. Consequently, from the velocity moment estimate in Lemma \ref{lem:jregularity} and energy-dissipation law \eqref{eq:abstract_energy_law}, $\bm j$ has the spatial regularity given in $L^{6/5}$. Interestingly, in view of Sobolev embedding, this is a mutual critical exponent for the coupling in the LLG equation $(\bm j \cdot \nabla) \bm m\,$. In contrast to \cite{DiPernaLions1989}, the proof of Theorem \ref{thm:glavni} fails in the critical case due to the lack of strong convergence of $\bm j$ in $L^{6/5}$.

\item Another open question is uniqueness, which is unknown even for weak solutions of the Vlasov-Maxwell system alone. The regularity theory for the LLG part is strong enough to obtain a partial uniqueness result, even for the critical exponent $L^{6/5}$.
\end{enumerate}

\begin{thm}\label{thm:uniqueness}
Let $\bm j \in L^\infty ((0,\infty); L^{6/5}(\mathbb R^3;\mathbb R^3))$ be fixed. Then the distributional solution to equation \eqref{eq:LLG} with regularity from the previous Theorem 
\begin{align*}
    \bm m \in L^\infty((0,\infty);H^2 (\mathbb R^3;\mathbb S^2)) \cap \dot H^1((0,\infty);L^2(\mathbb R^3;\mathbb R^3))
\end{align*}
is unique.
\end{thm}

\paragraph{Further questions and possible extensions}

\begin{enumerate}
\item The system of equations considered in this work are perhaps the most basic model that
features a coupling of dissipative magnetization dynamics and classical electron transport via emergent electromagnetic fields. The compactness methods of
\cite{DiPernaLions1989} and thus the present result extends to 
the relativistic case where $\bm v$ is replaced by $\bm v/\sqrt{1+|\bm v|^2}$, see also \cite{Rein2004, Glassey1986} for existence results in this case. The classical solutions can be extended globally in time provided the momentum support can be controlled. 

\item A weightier generalization towards a more realistic model of electron transport in solids lies in the inclusion of particle interactions in the form of a collisions by means of a Boltzmann operator or BGK model, a simplified form of it. Ignoring magnetic fields, global weak solution to the Vlasov-Poisson-BGK system have been constructed in \cite{Zhang2010,Zhang2013}. Due to the limited regularity of the Lorentz force, however, the arguments based on velocity moments lemmata as in \cite{Perthame1989, Perthame1990} do not extend to the Vlasov-Maxwell-BGK system, an open problem of its own.

\item In the context of the Landau-Litshitz-Gilbert equations, we neglect the coupling of the magnetization field to the Maxwell equations being part of the magnetic field, 
i.e. a constitutive law $\bm B= \bm H + \bm m$, giving rise to the Landau-Lifshitz-Gilbert-Maxwell system. In micromagnetics, it is customary to assume a quasistatic situation where electric fields are ignored and the magnetic Gau{\ss} law
gives rise to the so-called stray-field interaction, which is a non-local but lower order contribution, see e.g. \cite{Melcher2010} an literature therein. 

\item It would also be interesting to further investigate the role of Gilbert damping $\alpha>0$. The corresponding space-time $L^2$ bound on $\partial_t \bm m$
provides a suitable bound for the emergent electric field. The lack of a natural uniform bound in the case $\alpha=0$ requires high regularity of $f$, and it would be interesting to investigate local well-posedness results for this fully conservative system.

\end{enumerate}

\section{Solving the Landau-Lifshitz-Gilbert equation}
We examine global solvability of \eqref{eq:LLG_fourth} for a fixed current density $\bm j$ with regularity specified below. The requisite higher order Sobolev estimates can eventually be reduced to an $H^2$ estimate which is bounded by the energy \eqref{eq:frustratedmagnet}
for large enough $h$.

\begin{lem}\label{lem:energija}
Suppose $h>1/4$. Then $E(\bm m)$ is equivalent to $\|\bm m- \bm{\hat {e}}_3 \|_{H^2}^2\,.$
\end{lem}
\begin{proof}
The upper bound is straightforward. To obtain a lower bound, using Young's inequality for arbitrary $\epsilon>0$ we have
\begin{align*}
\|\nabla \bm m \|_{L^2}^2 = \int_{\mathbb R^3} |\xi|^2 \left| \mathcal F (\bm m - \bm{\hat {e}}_3) \right|^2 \, d\bm \xi \leq \int_{\mathbb R^3} \left(\frac{|\xi|^4}{4\epsilon} + \epsilon \right) \left| \mathcal F (\bm m - \bm{\hat {e}}_3) \right|^2 \, d\bm \xi \, .
\end{align*}
Then it follows
\begin{align*}
E(\bm m ) \geq \left(1-\frac{1}{4\epsilon}\right)\| \nabla^2 \bm m \|^2_{L^2} + (h-\epsilon) \| \bm m - \bm{\hat {e}}_3 \|_{L^2}^2 \, .
\end{align*}
Since $h>1/4$ we can take any $\epsilon>0$ such that $1/4<\epsilon<h$ and conclude.
\end{proof}

$H^2$ coercivity of the energy is closely related to uniform parabolicity of the governing Landau-Lifshitz-Gilbert equation.
To highlight the structure of \eqref{eq:LLG} as a fourth order parabolic system, we pass to the so-called Landau-Lifshitz formulation, see e.g. \cite{MelcherPtashnyk2013}. Extracting the leading fourth order terms yields
\begin{equation} \label{eq:LLG_fourth}
    (1+\alpha^2) \partial_t \bm m +   A(\bm m) \Delta^2 \bm m = A(\bm m) \bm f - \alpha \Lambda \bm m \, ,
\end{equation}
where $A(\bm m) \in \mathbb{R}^{3 \times 3}$ is such that for all $\bm \xi \in \mathbb{R}^3$
\[
A(\bm m)  \bm \xi =  \alpha \bm \xi - \bm m \times \bm \xi.
\]
The vector field $\bm f$ and the function $\Lambda$ depend on $\bm m$ and its derivatives. More precisely
\begin{equation*} 
\bm f =  \left[  h \, \bm{\hat {e}}_3 - \bm m \times (\bm j \cdot \nabla) \bm m - \Delta \bm m \right]^{\rm tan} \, ,
\end{equation*}
where $\bm \xi^{\mathrm tan}= \bm \xi - (\bm \xi \cdot \bm m) \bm m$. Moreover $\Lambda = -\bm m \cdot \Delta^2 \bm m$. 
Taking into account that $|\bm m|=1$ 
\begin{equation} 
\Lambda = |\Delta \bm m|^2 + \Delta | \nabla \bm m|^2 + 2 \nabla \bm m \cdot \nabla \Delta \bm m \, .
\end{equation}

\begin{lem} Suppose that, for $T>0$ and $l\ge 4$, 
\[
\bm m\in C([0,T]; H^l(\mathbb R^3;\mathbb S^2)) \quad \text{with} \quad \partial_t \bm m \in C([0,T];H^{l-4}(\mathbb R^3;\mathbb R^3)) \, ,
\]
is a solution of \eqref{eq:LLG}. Then
\begin{equation} \label{eq:LLG_energy_1}
E(\bm m (T))  + \alpha \int_0^T \| \partial_t \bm m \|_{L^2}^2 \, dt = E(\bm m_0) -\int_0^T \langle \bm j, \bm e \rangle dt.
\end{equation}
and, for $\lambda=\alpha/(1+\alpha^2)$, 
\begin{align} \label{eq:LLG_energy_2}
  \|\Delta \bm m(T)\|_{L^2}^2  + \lambda \int_0^T \|\Delta^2 \bm m\|_{L^2}^2 \, dt
\le\|\Delta \bm m_0\|_{L^2}^2 + C\, &\int_0^T  \alpha^{-1}  \| \bm f\|_{L^2}^2 + \lambda\| \Lambda \|_{L^2}^2 \, dt \, . 
\end{align}
\end{lem}

\begin{proof}
\eqref{eq:LLG_energy_1} is obtained upon multiplying \eqref{eq:LLG} by $\bm m \times \partial_t \bm m$. \eqref{eq:LLG_energy_2} is obtained upon multiplying \eqref{eq:LLG_fourth} by $\Delta^2 \bm m$, using Young's inequality and
the fact that $|A \bm \xi|^2=(1+\alpha^2)|\bm \xi|^2$.
\end{proof}

\paragraph{Short-time solutions} Fourth order quasilinear parabolic systems of the more general form
$\partial_t \bm u +A(\bm u)  \Delta^2 \bm u = \bm B(t, \bm x, \bm u, \dots, \nabla^3 \bm u)$ 
admit a local theory of existence theory in Sobolev spaces $H^l(\mathbb{R}^3; \mathbb{R}^3)$ with $l\ge 5$ so that, by Sobolev embedding, $\nabla^k \bm u$ is uniformly bounded in space for $0 \le k \le 3$.
Assuming that 
$A$ is smooth and uniformly elliptic, i.e., there exists $\alpha>0$ such that $\bm \xi A(\bm u) \bm \xi \ge \alpha |\bm \xi|^2$ for all $\bm u$ and $\bm \xi$, and $\bm B$ is continuous with smooth dependence on $\bm u$ and its derivatives and local bounds that are independent of $x$ and $t$, a priori estimates 
are obtained by using multipliers $\Delta^k \bm u$ as in \eqref{eq:LLG_energy_2} but for $1 \le k \le l$. A bootstrap and Gronwall-type argument yields a $H^l$-bound up to some time $T>0$. Analogue bounds 
can be obtained for suitably approximated or truncated systems that can be solved locally by a ODE argument. Compactness arguments then yield a short time solution to the original system as in \cite{Taylor3}. Details can be found in \cite{Tvrtko_Diss}. Letting $\bm u = \bm m - \bm{\hat {e}}_3$
as is \cite{Melcher2014, MelcherPtashnyk2013}, this modified Galerkin method applies to \eqref{eq:LLG} and yields short time solutions $\bm m\in C([0,T]; H^l(\mathbb R^3;\mathbb S^2)) \cap C^1([0,T]; H^{l-4}(\mathbb R^3;\mathbb S^2))$.

\subsubsection*{Global smooth solution}
Owing to the special structure of the geometric nonlinearities of \eqref{eq:LLG}, uniform bounds extend to all times. 

\begin{thm}\label{thm:LLGglobalno}
Suppose $\bm m_0 \in H^l(\mathbb R^3;\mathbb S^2)$ and $\bm j \in C([0,\infty );H^{l-2}(\mathbb R^3;\mathbb R^3))$ for some integer $l\geq 5\,$. Then there exists a unique global solution of \eqref{eq:LLG_fourth} such that
\begin{equation*}
     \bm m\in C([0,\infty); H^l(\mathbb R^3;\mathbb S^2)) \quad \text{and} \quad \partial_t \bm m \in C([0,\infty);H^{l-4}(\mathbb R^3;\mathbb R^3))\, .
\end{equation*}
\end{thm}
\begin{rem}
In Theorem \ref{thm:uniqueness} we assume less regularity on $\bm j$ and thus, in particular, it applies to obtain uniqueness for the solution from Theorem \ref{thm:LLGglobalno}.
\end{rem}
\begin{proof}[\textbf{Proof of Theorem \ref{thm:LLGglobalno}}]
We proceed in three steps. We first use Gronwall's inequality on the energy inequality \eqref{eq:LLG_energy_1}. Then we write down estimates that justify the use of Gronwall's inequality in \eqref{eq:LLG_energy_2} which proves $\Delta^2 \bm m$ remains bounded in $L^2((0,T)\times \mathbb R^3;\mathbb R^3)$ for all $T>0$. Finally, we obtain global extension. 

\textbf{\textit{Step 1}} From \eqref{eq:LLG_energy_1}, Lemma \ref{lem:energija} and Young's inequality we obtain
\begin{align*}
    \frac{\alpha}{2} \, \int_0^T \|\partial_t \bm m\|_{L^2}^2 \, dt + C_1\, \| \bm m(T) - \bm{\hat {e}}_3\|_{H^2}^2 \leq  E(\bm m_0) +  C_2 \, \|\bm j\|_{L^\infty_{t,x}}^2 \int_0^T \| \nabla \bm m \|^2_{L^2} \, dt \, ,
\end{align*}
for some constants $C_1,C_2>0$. Using Gronwall's inequality we get
\begin{align}\label{eq:energyStep1}
    \frac{\alpha}{2} \, \int_0^T \|\partial_t \bm m\|_{L^2}^2 \, dt + C_1 \,  \| \bm m(T) - \bm{\hat {e}}_3\|_{H^2}^2 \leq E(\bm m_0) \, e^{C(T,\bm j) } \, ,
\end{align}
where $C(T,\bm j) = C_2/C_1 \,T \, \|\bm j\|_{L^\infty_{t,x}}^2 \, .$

\textbf{\textit{Step 2}} We expand on \eqref{eq:LLG_energy_2}. To estimate $\|\Lambda\|_{L^2}$ we use the following Sobolev and interpolation inequalities
\begin{align*}
    \| f \|_{L^4} \leq C \, \| f\|_{\dot H^{3/4}} \, , \quad
    \| f\|_{\dot H^s} \leq \| f\|_{\dot H^{s_1}}^{(s_2-s)/(s_2-s_1)} \, \| f\|_{\dot H^{s_2}}^{(s-s_1)/(s_2-s_1)} \, ,
\end{align*}
where $0\leq s_1<s<s_2$. For the first term of $\Lambda$ we have
\begin{align*}
    \| |\Delta \bm m|^2 \|_{L^2} &\leq C \, \| \Delta \bm m \|_{\dot H^{3/4}}^2 
    \leq C \, \| \Delta \bm m \|^{3/4}_{\dot H^{2}} \| \Delta \bm m \|_{L^2}^{5/4}  
    \leq C \, \| \Delta^2 \bm m \|_{L^2}^{3/4} \| \bm m - \bm{\hat {e}}_3 \|_{H^2}^{5/4} \, .
\end{align*}
The second term of $\Lambda$ is bounded by
\begin{align*}
    \| \Delta |\nabla \bm m|^2\|_{L^2} \leq C \left( \| |\nabla^2 \bm m |^2 \|_{L^2} +  \| \nabla \bm m \cdot \nabla \Delta  \bm m \|_{L^2}  \right) \, ,
\end{align*}
where $\| |\nabla^2 \bm m |^2 \|_{L^2}$ satisfies the same estimate as $\| |\Delta \bm m|^2 \|_{L^2}$. 
Moreover,
\begin{align*}
    \| \nabla \bm m \cdot \Delta \nabla \bm m \|_{L^2} &\leq C \, \| \nabla \bm m \|_{L^4} \| \Delta \nabla \bm m \|_{L^4} \\
    &\leq C \, \| \nabla \bm m \|_{\dot H^{3/4}} \| \Delta \nabla \bm m \|_{\dot H^{3/4}} \\
    &\leq C \, \| \bm m - \bm{\hat {e}}_3\|_{H^2} \| \Delta \bm m \|_{\dot H^{7/4}} \\
    &\leq C \, \| \bm m - \bm{\hat {e}}_3\|_{H^2} \| \Delta \bm m \|_{L^2}^{1/8} \| \Delta \bm m \|_{\dot H^2}^{7/8} \\
    &\leq C \, \| \bm m - \bm{\hat {e}}_3 \|_{H^2}^{9/8}  \| \Delta^2 \bm m \|^{7/8}_{L^2} \, , 
\end{align*}
which provides the required estimate for the last term of $\Lambda$ as well. We estimate $\|\bm f\|_{L^2}$ by
\begin{align*}
    \| \bm f \|_{L^2} \leq h \|  \bm{\hat {e}}_3^{\rm tan} \|_{L^2} + \| (\bm j \cdot \nabla ) \bm m \|_{L^2} + \|\Delta \bm m \|_{L^2} \, ,
\end{align*}
where we used $|\bm \xi ^{\rm tan} | \leq |\bm \xi|.$ Since $|\bm m| = 1$ we have $2(1-m_3) = |\bm m-  \bm{\hat {e}}_3|^2$ and hence
\begin{align*}
   \|  \bm{\hat {e}}_3^{\rm tan} \|_{L^2} = \left(\int_{\mathbb R^3} 1-m_3^2 \, d\bm x \right)^{1/2} = \left(\int_{\mathbb R^3} \frac{1}{2}|\bm m - \bm{\hat {e}}_3 |^2 (1+m_3)  \, d\bm x \right)^{1/2} \leq \|\bm m - \bm{\hat {e}}_3\|_{L^2} \, . 
\end{align*}
The term including $\bm j$ is bounded by
\begin{align*}
 \| (\bm j \cdot \nabla ) \bm m \|_{L^2}  \leq  \| \bm j \|_{L^\infty_{t,x}} \|\nabla \bm m \|_{L^2} \, . 
\end{align*}
Therefore 
\begin{align*}
    \| \bm f \|_{L^2} \leq C \, (1+\| \bm j\|_{L^\infty_{t,x}} ) \| \bm m - \bm{\hat {e}}_3\|_{H^2} 
\end{align*}
for a constant only depending on $h$.
Using the estimates above in \eqref{eq:LLG_energy_2} we get
\begin{align*}
    \| \Delta \bm m (T) \|^2_{L^2} + \lambda \int_0^T  \|\Delta^2\bm m\|^2_{L^2} \, dt \leq \| \Delta \bm m_0 \|^2_{L^2} + C &\int_0^T R\big(\| \bm m - \bm{\hat {e}}_3\|_{H^2},\| \Delta^2\bm m \|_{L^2}\big) \\ 
    \ &\phantom{{} =  }+ \| \bm j \|_{L^\infty_{t,x}}^2 \|\bm m - \bm{\hat {e}}_3 \|_{H^2}^2 \, dt \,  ,
\end{align*}
where the function $R(a,b)$ is given by
\begin{align*}
    R(a,b) = a^2 + a^4 + a^{5/2} b^{3/2} + a^{9/4} b^{7/4}.
\end{align*}
Using Young's inequality and absorbing the highest order term on the left-hand side leads to
\begin{align*}
     \| \Delta \bm m (T) \|^2_{L^2} + \frac{\lambda}{2} \int_0^T  \|\Delta^2\bm m\|^2_{L^2} \, dt \leq \| \Delta \bm m_0\|^2_{L^2} + C &\int_0^T 1+ \| \bm m - \bm{\hat {e}}_3\|_{H^2}^{18}  \\
     \ &\phantom{{} =  } +\| \bm j \|_{L^\infty_{t,x}}^2 \| \bm m - \bm{\hat {e}}_3\|_{H^2}^2 \, dt \, .
\end{align*}
Using the $H^2$ estimate \eqref{eq:energyStep1}, we obtain
\begin{align}\label{eq:energyStep2}
      \| \Delta \bm m(T) \|^2_{L^2} + \frac{\lambda}{2} \, \int_0^T  \|\Delta^2\bm m\|^2_{L^2} \, dt \leq  \| \Delta \bm m_0\|_{L^2}^2 + C(E(\bm m_0))\, T \, e^{C(T,\bm j)} \left(1+\| \bm j \|_{L^\infty_{t,x}}^2 \right) \, ,
\end{align}
where $C(T,\bm j) = C_2/C_1 \, T  \, \|\bm j\|_{L^\infty_{t,x}}^2 \, .$

\textbf{\textit{Step 3}} We show that the estimates from the previous two steps imply uniform bounds on higher
order Sobolev norms of $\bm m - \bm{\hat {e}}_3$. To this end, we use multipliers $\Delta^k \bm m$ for all $1\leq k \leq l$ and integrate by parts, i.e., letting $D^k = \nabla \otimes \dots \otimes \nabla$ the $k$ fold tensor product, we apply $D^k$ to \eqref{eq:LLG_fourth} and integrate against $D^k \bm m$. We focus on the highest order terms, using Moser's inequality 
$\|fg\|_{H^k} \le C \left( \|f\|_{H^k} \|g\|_{L^\infty}+ \|f\|_{L^\infty} \|g\|_{H^k} \right)$
as an additional tool.
We estimate the first term coming from $\Lambda$ by
\begin{align}\label{eq:LLGStep3_21}
    \big \langle D^k(\bm m \, |\Delta \bm m |^2),D^k \bm m \big \rangle &\leq \| |\Delta \bm m|^2 \bm m \|_{H^l} \| \bm m - \bm {\hat e}_3\|_{H^l} \\
    &\leq C \, \left ( \| \Delta \bm m \|_{L^\infty}^2 \| \bm m - \bm {\hat e}_3\|_{H^l} + \| \Delta \bm m \|_{L^\infty} \| \Delta \bm m \|_{H^l} \right) \| \bm m - \bm {\hat e}_3 \|_{H^l} \nonumber \, .
\end{align}
To estimate the second term from $\Lambda$ we first rewrite it as in Step 2
\begin{align*}
    \big \langle D^k( \bm m \, \Delta |\nabla \bm m|^2), D^k \bm m \big \rangle = \big \langle D^k(\bm m |\nabla^2 \bm m |^2 ) + D^k \big(\bm m (\nabla \bm m \cdot \nabla \Delta \bm m )\big), D^k \bm m \rangle \eqdef (\RN{1})+ (\RN{2}) \, .
\end{align*}
We estimate $(\RN{1})$ in the same way as \eqref{eq:LLGStep3_21} to obtain the same bound.
To estimate $(\RN{2})$ we need to integrate by parts to obtain
\begin{alignat*}{3}
\big \langle D^k \big(\bm m (\nabla \bm m \cdot \nabla \Delta \bm m )\big) , D^k \bm m \big \rangle &= &&- \big \langle D^k (\bm m \, |\Delta \bm m |^2 ), D^k \bm m \big \rangle \\
& &&  - \big \langle D^k \big( (\nabla \bm m \cdot \Delta \bm m)\big) \nabla \bm m, D^k \bm m\big \rangle \\
& &&- \big \langle D^k \big(\bm m  (\nabla \bm m \cdot \Delta \bm m ) \big), \nabla D^k \bm m \big \rangle \\
&\eqdef && \, (a)+(b)+(c) \, .
\end{alignat*}
The first term $(a)$ we already estimated in \eqref{eq:LLGStep3_21}. We estimate $(b)$ as follows
\begin{align*}
    |(b)| &\leq C \,\|\bm m - \bm{\hat {e}}_3\|_{H^l} \left(\| \nabla \bm m\|^2_{L^\infty}\|\Delta \bm m \|_{H^l}  + \|\nabla \bm m \|_{L^\infty} \|\Delta \bm m \|_{L^\infty} \|\nabla \bm m \|_{H^l} \right)\, .
\end{align*}
In the same way we get
\begin{align*}
|(c)|\leq C\, \| \nabla \bm m \|_{H^l} &\Big( \| \Delta \bm m \|_{L^\infty} \| \nabla \bm m \|_{H^l} + \| \nabla \bm m \|_{L^\infty} \| \Delta \bm m \|_{H^l} \\
\ &\phantom{{} +  }+ \| \nabla \bm m \|_{L^\infty} \|\Delta \bm m \|_{L^\infty} \|\bm m - \bm{\hat {e}}_3\|_{H^l} \Big). 
\end{align*}
Remark that we have already estimated the last term coming from $\Lambda$ in $(\RN{2})$.

To treat the leading order term we have
\begin{align*}
    \big \langle D^k \left(A(\bm m) \Delta^2 \bm m \right) , D^k \bm m \big \rangle = \alpha \, \|D^k \Delta \bm m \|_{L^2}^2 + \big \langle D^k ( \bm m \times \Delta^2 \bm m), D^k \bm m \big \rangle \, .
\end{align*}
To estimate the second term we first integrate by parts to get
\begin{align*}
    \big \langle D^k ( \bm m \times \Delta^2 \bm m), D^k \bm m \big \rangle = 2\, \big \langle D^k (\nabla \bm m \times \Delta \bm m ) , \nabla D^k \bm m \big \rangle + \big \langle D^k (\bm m \times \Delta \bm m ), \Delta D^k \bm m \big \rangle \, .
\end{align*}
We only need to estimate the second term on the right-hand side since the first one is estimated analogously to (c). Since $\langle \bm m \times D^k \Delta \bm m, D^k \Delta \bm m \rangle= 0$ we have
\begin{align*}
    \big \langle D^k (\bm m \times \Delta \bm m ), \Delta D^k \bm m \big \rangle &= \big \langle D^{k-1} (D \bm m \times \Delta \bm m ), \Delta D^k \bm m \big \rangle \\
    &\leq C\, \| \Delta \bm m \|_{H^l} \left( \| \nabla \bm m \|_{L^\infty} \|\Delta \bm m \|_{H^{l-1}} + \|\Delta \bm m \|_{L^\infty} \| \bm m - \bm{\hat {e}}_3\|_{H^l} \right) \, .
\end{align*}
We have now estimated all of the highest order terms. We now estimate terms coming from $\bm f$. We have
\begin{align*}
    \left \langle D^k \left[A(\bm m) (\Delta \bm m )^{\rm tan} \right] , D^k \bm m \right \rangle &\leq C\, \|\bm m - \bm{\hat {e}}_3\|_{H^l} \left( \| \Delta \bm m \|_{H^l}  + \|\Delta \bm m \|_{L^\infty} \| \bm m - \bm{\hat {e}}_3\|_{H^l} \right) \, , \\
    \left \langle D^k \left[ A(\bm m) (h \, \bm{\hat {e}}_3)^{\rm tan} \right] , D^k \bm m \right \rangle &\leq C \, \| \bm m - \bm{\hat {e}}_3\|_{H^l}^2 \, .
\end{align*}
The terms including the current density for $k\geq 2$ we estimate as follows
\begin{align*}
    \big \langle D^k \left[ A(\bm m) \, (\bm m \times (\bm j \cdot \nabla ) \bm m) \right], D^k \bm m \big \rangle &= \big \langle D^{k-2} \left[ A(\bm m) \, (\bm m \times  (\bm j \cdot \nabla ) \bm m) \right], D^{k-2} \Delta^2 \bm m \big \rangle \\
    &\leq  C \, \| \Delta \bm m \|_{H^l} \Big( \| \bm m - \bm{\hat {e}}_3\|_{H^{l-2}} \| \bm j \|_{L^\infty} \| \nabla \bm m \|_{L^\infty} \\
    \ &\phantom{{} =  }  + \|\nabla \bm m\|_{L^\infty} \|\bm j \|_{H^{l-2}} + \|\nabla \bm m\|_{H^{l-2}} \|\bm j \|_{L^\infty} \Big) \, .
\end{align*}
For $k=1$ we simply have
\begin{align*}
    \big \langle D \left[ A(\bm m) \, (\bm m \times (\bm j \cdot \nabla ) \bm m) \right], D \bm m \big \rangle &= -\big \langle  \left[ A(\bm m) \, (\bm m \times (\bm j \cdot \nabla ) \bm m) \right], D^2 \bm m \big \rangle \\ &\leq C \, \|\bm j\|_{L^\infty} \| \bm m - \bm{\hat {e}}_3\|_{H^l}^2 \, .
\end{align*}
Making use of the interpolation inequality $\|D f \|_{H^l} \leq C\,\|D^2 f \|_{H^{l}}^{1/2} \| f \|_{H^{l}}^{1/2}\,$ we have
\begin{align}\label{eq:interpolation}
    \|\nabla \bm m \|_{H^l} \leq C\, \|\bm m - \bm{\hat{e}}_3 \|_{H^{l}}^{1/2} \| \Delta \bm m \|_{H^l}^{1/2}\,\quad \text{and} \quad \| \Delta \bm m \|_{H^{l-1}} \leq C\, \|\bm m - \bm{\hat{e}}_3 \|_{H^l}^{1/2} \| \Delta \bm m \|_{H^l}^{1/2} \, .
\end{align}
Summing up the above estimates over all $k \leq l$, using \eqref{eq:interpolation}, Young's inequality, Sobolev embedding and absorbing the highest order term on the left-hand side we obtain 
\begin{align*}
   \|\bm m(T) - \bm{\hat {e}}_3\|_{H^l}^2 + \lambda \int_0^T  \| \Delta \bm m \|_{H^l}^2 \, dt \leq \| \bm m_0- \bm{\hat {e}}_3 \|_{H^l}^2 + C \int_0^T \| \bm m - \bm{\hat {e}}_3\|_{H^l}^2\, F(\bm m , \bm j ) \, dt \, ,
\end{align*}
where
\begin{align*}
    F(\bm m , \bm j ) = 1+ \| \Delta \bm m\|_{L^\infty}^2 + \| \nabla \bm m \|_{L^\infty}^4 + \|\nabla \bm m \|_{L^\infty}^{4/3} \| \Delta \bm m \|_{L^\infty}^{4/3}  
    + \big(1+ \| \nabla \bm m \|^2_{L^\infty}\big) \| \bm j \|_{H^{l-2}}^2 \, ,
\end{align*}
and lower order terms are taken into account as well.
Using Gronwall's inequality we get
\begin{align}\label{eq:Gronwallograda}
    \| \bm m (T) - \bm{\hat {e}}_3\|_{H^l}^2 \leq \| \bm m_0- \bm{\hat {e}}_3 \|_{H^l}^2 \, e^{C(T,\bm m, \bm j)} \, ,
\end{align}
where
\begin{align*}
C(T,\bm m,\bm j) = C \int_0^T F(\bm m, \bm j) \, dt \, .
\end{align*}
We now turn to estimating $F(\bm m, \bm j)$ by using inequality \eqref{eq:energyStep2}. Let $3/2 < s \leq 2$ and $q$ a real number such that $q \,s = 4$. Then we have
\begin{align*}
    \int_0^T \| \Delta \bm m\|_{L^\infty}^q \, dt \leq C  \int_0^T \| \Delta \bm m \|_{\dot H^s}^q \, dt &\leq  C \int_0^T \| \Delta \bm m\|_{\dot H^2}^{qs/2} \|\Delta \bm m\|_{L^2}^{(2-s)q/2} \, dt \\
    &\leq C\, \| \Delta \bm m \|^{(2-s)q/2}_{L^\infty_t L^2_x} \, \int_0^T \| \Delta^2 \bm m  \|_{L^2}^{2} \, dt \\
    &\leq C(T,\bm j, E(\bm m_0)) \, ,
\end{align*}
where $ C(T,\bm j, E(\bm m_0))$ is a constant coming from \eqref{eq:energyStep2}. We have therefore obtained $\Delta \bm m \in L^q_t L^\infty_x$ for all $ 1\leq q < 8/3$\,. Similarly let $3/2<s \leq 2$ and $q$ a real number such that $q (s-1)/2 = 2$, then
\begin{align*}
    \int_0^T \| \nabla \bm m \|_{L^\infty}^q \, dt \leq C \int_0^T \| \nabla \bm m\|_{\dot H^s}^q \, dt &\leq C \int_0^T  \|\nabla \bm m\|_{\dot H^1}^{(3-s)q/2} \|\nabla \bm m \|_{\dot H^3}^{q(s-1)/2} \, dt \\
    &\leq C \,\| \Delta \bm m \|_{L^\infty_t L^2_x}^{(3-s)q/2} \int_0^T \| \Delta^2 \bm m \|_{L^2}^2 \, dt \\
    &\leq C(T,\bm j ,E(\bm m_0)) \, ,
\end{align*}
where again $ C(T,\bm j, E(\bm m_0))$ is the appropriate constant coming from \eqref{eq:energyStep2}. We have then obtained $\nabla \bm m \in L^q_t L^\infty_x$ for all $1\leq q < 8\,$. It remains to bound the mixed term with the help of Holder's inequality
\begin{align*}
    \int_0^T \| \nabla \bm m \|_{L^\infty}^{4/3} \| \Delta \bm m  \|_{L^\infty}^{4/3} \, dt \leq \left(\int_0^T \| \nabla \bm m  \|_{L^\infty}^4 \, dt\right)^{1/3} \left(\int_0^T \| \Delta \bm m  \|_{L^\infty}^2\, dt \right)^{2/3} \, .
\end{align*}
Going back to inequality \eqref{eq:Gronwallograda} we obtain the bound
\begin{align}\label{eq:Gronwallograda1}
\| \bm m(T)-\bm{\hat {e}}_3\|_{H^l} \leq C(T, \bm j  , E(\bm m_0)) \, ,
\end{align}
for all times $T>0$ which gives us a global solution.
\end{proof}

\section{Tools for transport equations}\label{subsection:VM}
We summarize some well known methods and results from the theory of transport equations that will be necessary in our analysis afterwards.

\subsubsection*{Characteristic flow}
 One of the main tools in the topic of kinetic equations is the theory of characteristic flow. In the case of a smooth and divergence-free vector field, the initial distribution gets transported along the characteristics. In particular, let us take the Vlasov equation 
 \begin{align}\label{eq:Vlasovkarakteristike}
     \partial_t f + \bm v \cdot \nabla_x f + \left( \bm F_1(t,\bm x) + \bm v \times \bm F_2(t,\bm x) \right) \cdot \nabla_v f = 0 \, ,
 \end{align}
 for some bounded functions $\bm F_1,\bm F_2\in C([0,T]\times \mathbb R^3 ;\mathbb R^3)$ which are continuously differentiable with respect to $\bm x$. Then for every $t\in [0,T]$ and $(\bm x,\bm v) \in \mathbb R^3 \times \mathbb R^3$ there exists a unique solution $[0,T] \ni s \rightarrow (\bm X,\bm V) (s,t,\bm x,\bm v)$ to the characteristic system of ODEs
 \begin{align*}
    \begin{cases}
    \dot{\bm X}(s) = \bm V(s), \quad &\bm X(t) = \bm x \, , \\
    \dot{\bm V}(s) = \bm F_1(s,\bm X(s)) + \bm V(s) \times  \bm F_2(s,\bm X(s)), \quad &\bm V(t) = \bm v \, .
    \end{cases}
\end{align*}
The characteristic flow is volume preserving as the generating vector field \[
\bm u(t,\bm x, \bm v)=(\bm v,\bm F_1(t ,\bm x) + \bm v \times  \bm F_2(t,\bm x))\, ,\] 
is divergence-free in $(\bm x, \bm v)$. Note that by means of this vector field, the Vlasov equation \eqref{eq:Vlasovkarakteristike} can be recast into the linear transport equation
\begin{equation}\label{eq:Vlasovtransport}
\partial_t f +  \nabla_{x,v} \cdot (\bm u f) = 0.
\end{equation}
For smooth initial data $f_0\in C^1(\mathbb R^3\times \mathbb R^3)$ it is well known that the function given by
\begin{align*}
    f(t,\bm x,\bm v) \defeq f_0(\bm X(0,t,\bm x,\bm v), \bm V(0,t,\bm x, \bm v)), \quad t\in [0,T],(\bm x,\bm v) \in \mathbb R^3\times \mathbb R^3 \, ,
\end{align*}
is the unique solution to \eqref{eq:Vlasovkarakteristike} in the space $C^1([0,T]\times \mathbb R^3 \times \mathbb R^3)$. The solution is constant along every solution of the characteristic system. Moreover, if $f_0$ is non-negative then so is $f$. Finally, by the volume preservation of the characteristic flow,
$f$ satisfies the $L^p$ conservation property
\begin{align*}
    \|f(t)\|_{L^p} =\|f_0\|_{L^p}, \quad t\in [0,T],\; p \in [1,\infty] \, .
\end{align*}
 We refer to \cite{ReinKineticElsevier} for the proof of these results.
\subsubsection*{Velocity averaging}

Let $f\in L^2(\mathbb R \times \mathbb R^N\times \mathbb R^N)$ satisfy the transport equation
\begin{equation} \label{eq:transport}
    \partial_t f + \bm v \cdot \nabla_x f = g \quad \text{in} \quad \mathscr{D}'(\mathbb R\times \mathbb R^N\times \mathbb R^N).
\end{equation}
Depending on the reguarity of the distribution $g$, local averages in $\bm v$ satisfy improve regularity properties in fractional Sobolev spaces in 
space-time:

\begin{thm}[DiPerna, Lions \cite{DiPernaLions1989}]\label{thm:dipernalions}
Let $m$ be a non-negative integer, let $f\in L^2 (\mathbb R\times \mathbb R^N \times \mathbb R^N)$ satisfy \eqref{eq:transport} where $g$ is given by
\begin{equation*}
    g= g_1 + D_v^m g_2
\end{equation*}
and $g_1,\, g_2 \in L^q(B_R \, ; L^p(\mathbb R\times \mathbb R^N_x))$ for all $R<\infty$, where $2(N+1)/(N+3) < p \leq 2$ and $1\leq q \leq 2$. Then,
\begin{equation*}
    \int_{\mathbb R^N} f(t,\bm x , \bm v) \psi(\bm v) \, d\bm v \in H^s(\mathbb R \times \mathbb R^N) \quad \text{for each} \quad \psi \in \mathscr D (\mathbb R^N) \, ,
\end{equation*}
where $s=1/2 (1-\theta) (m+1/q+1/2)^{-1}$ and $\theta = (N+1)(2-p) /2p$.
\end{thm}
\subsubsection*{Velocity moment estimates}

To bound the current density $\bm j$ uniformly we make use of the
velocity moment estimate from kinetic theory. We refer to Lemma 1.8 in \cite{ReinKineticElsevier} for the proof of this result.
\begin{lem} \label{lem:jregularity}
For $k\geq 0$ we denote the kth order moment density and the kth order moment in velocity of a nonnegative, measurable function $f:\mathbb R^6\rightarrow [0,\infty)$ by
\begin{align*}
    m_k(f) (\bm x) \defeq \int_{\mathbb R^3} |\bm v|^k \, f(\bm x,\bm v) \, d \bm v \, ,
\end{align*}
and
\begin{align*}
    M_k(f) \defeq \int_{\mathbb R^3} m_k(f) (\bm x) \, d\bm x = \int \int_{\mathbb R^3 \times \mathbb R^3} |\bm v|^k f(\bm x,\bm v)\, d\bm v \, d\bm x \, .
\end{align*}
Let $1\leq p,q\leq \infty$ with $1/p+1/q = 1$, $0\leq k'\leq k <\infty$ and 
\begin{align*}
    \ell \defeq \frac{k+3/q}{k'+3/q+(k-k')/p} \, .
\end{align*}
If $f\in L^p_+ (\mathbb R^6)$ with $M_k(f)<\infty$ then $m_{k'} (f) \in L^\ell(\mathbb R^3)$ and
\begin{align*}
\| m_{k'}(f)\|_{L^\ell} \leq C \, \| f\|_{L^p}^{(k-k')/(k+3/q)} \, M_k(f)^{(k'+3/q)/(k+3/q)} \, ,
\end{align*}
where $C=C(k,k',p)>0$.
\end{lem}
\section{Proof of Theorem \ref{thm:glavni}} \label{section:Theorem1}

The arguments closely follow the strategy from \cite{DiPernaLions1989}, starting from a regularized system which admits global smooth solutions so that the energy estimate provides the requisite uniform bounds. This enables us to apply compactness arguments based on velocity averaging and renormalization to the extended Vlasov equation containing emergent electromagnetic field contribution.

\subsection*{Regularized LLG-VM system}
We first regularize the initial conditions for the VM system, i.e., we consider families 
$f_0^\varepsilon \in \mathscr{D}_+ (\mathbb R^3 \times \mathbb R^3)$ and $\bm E^\varepsilon_0, \, \bm B^\varepsilon_0 \in \mathscr{D}(\mathbb R^3;\mathbb R^3)$ so that
\begin{equation*}
    \int \int_{\mathbb R^3 \times \mathbb R^3} |f_0 - f_0^\varepsilon| (1+|\bm v|^2) + |f_0 - f_0^\varepsilon|^r \; d\bm x \, d\bm v \xrightarrow[]{\varepsilon} 0 \,
\end{equation*}
and
\begin{equation*}
    \int_{\mathbb R^3} |\bm E_0 - \bm E_0^\varepsilon|^2 + |\bm B_0 - \bm B_0^\varepsilon|^2 \, d\bm x \xrightarrow[]{\varepsilon} 0 \, .
\end{equation*}
Moreover, for an integer $l\ge 7$ there exists $\bm m_0^\varepsilon \in H^l(\mathbb R^3;\mathbb S^2)$ such that
\begin{equation*}
    \|\bm m_0 - \bm m_0^\varepsilon\|_{H^2} \xrightarrow[]{\varepsilon} 0 \,,
\end{equation*}
see e.g. \cite{Melcher2012}. Then a regularized system is obtained by
regularizing the current density and the Lorentz force by means of a 
suitable mollifier $K_\varepsilon$, i.e.
\begin{equation}\label{eq:LLG1}
    \partial_t \bm m^\varepsilon = \bm m^{\varepsilon} \times \left( \alpha \, \partial_t \bm m^\varepsilon + \Delta \bm m ^\varepsilon + \Delta^2 \bm m^\varepsilon - h\, \bm{\hat {e}}_3 \right)  - (K_\varepsilon \, \bm j^\varepsilon \cdot \nabla ) \bm m^\varepsilon \,,
\end{equation}
coupled to the regularized VM system
\begin{gather}
    \partial_t f^\varepsilon + \bm v \cdot \nabla_x  f^\varepsilon 
    + (K_\varepsilon \,\bm F^{\varepsilon}) \cdot \nabla_v  f^\varepsilon =0
    \, , \label{eq:Vlasov0} \\[1mm]
    \varepsilon_r \, \partial_t \bm E^\varepsilon - \frac{1}{\mu_r} \, \nabla \times \bm B^\varepsilon = - K_\varepsilon \, \bm j^\varepsilon \, , \quad \partial_t \bm B^\varepsilon + \nabla \times \bm E^\varepsilon= 0 \, , \label{eq:Maxwell}
\end{gather}
where
\begin{gather}
 \bm j^\varepsilon = -\int_{\mathbb R^3} f^\varepsilon \, \bm v \, d\bm v \, , \nonumber \\ 
 \bm F^{\varepsilon} = - ( \bm E^\varepsilon +   \bm e^\varepsilon + \bm v \times( \bm B^\varepsilon +   \bm b^\varepsilon)) \, , \nonumber \\
      e^\varepsilon_i = \bm m^\varepsilon \cdot (\partial_i \bm m^\varepsilon \times \partial_t \bm m^\varepsilon), \quad b^\varepsilon_i = \epsilon^{ijk}\, \bm m^\varepsilon \cdot (\partial_j \bm m^\varepsilon \times \partial_k \bm m^\varepsilon). \label{eq:regularizedemergent}
\end{gather}
With a slight abuse of notation, the operator $K_\varepsilon$ is a convolution operator defined by
\begin{equation*}
    K_\varepsilon f (\bm x) = \int_{\mathbb R^3} K_\varepsilon(\bm x-\bm y)\, f(\bm y) \, d\bm y \quad \text{for} \quad f\in L^1_{loc}(\mathbb R^3)\, ,
\end{equation*}
where $K_\varepsilon$ is a standard mollifier satisfying
\begin{gather*}
    K_\varepsilon \in C_c^\infty(\mathbb R^3),\quad \mathrm{supp}\, K_\varepsilon \subseteq \overline{ B_{\varepsilon}}, \quad \int_\mathbb{R^3} K_\varepsilon\, d\bm x = 1,\\ 
    \quad K_\varepsilon(\bm x) \geq 0 \quad \text{and} \quad K_\varepsilon(\bm x) = K_\varepsilon (-\bm x) \text{ for } \bm x\in \mathbb R^3 \, .
\end{gather*}
From $K_\varepsilon (\bm x) = K_\varepsilon (-\bm x)$ it follows that the operator $K_\varepsilon$ is self-adjoint with respect to the $L^2$ scalar product.

Equation \eqref{eq:LLG1} is convenient due to the divergence structure of the highest order term, i.e. we have
\begin{align}\label{eq:divergencestructure}
    \bm m^\varepsilon \times \Delta ^2 \bm m^\varepsilon &= \Delta (\bm m^\varepsilon \times \Delta \bm m^\varepsilon) - 2 \sum_{k=1}^3 \partial_k \bm m^\varepsilon \times \partial_k (\Delta \bm m^\varepsilon)\notag \\
    &= \Delta (\bm m^\varepsilon \times \Delta \bm m^\varepsilon) - 2 \sum_{k=1}^3 \partial_k \left(\partial_k \bm m^\varepsilon \times \Delta \bm m^\varepsilon \right)\, .
\end{align}

\subsection*{Short-time solution to the regularized system}
\begin{prop}
Let $\varepsilon > 0 $ be fixed. Let $l\geq 7$ be an integer, $R>0$ and suppose we have initial conditions
\begin{gather*}
    f_0^\varepsilon \in H^{l-2} (\mathbb R^3\times \mathbb R^3), \; \mathrm{supp} \, f_0^\varepsilon \subseteq B_R \times B_R \, ,\\
    \bm m^\varepsilon_0 \in H^l(\mathbb R^3;\mathbb S^2) \,, \; \bm E_0^\varepsilon, \; \bm B_0^\varepsilon \in H^{l-2}(\mathbb R^3;\mathbb R^3).
\end{gather*}
Then there exists $T^*>0$ and a local solution for \eqref{eq:LLG1}-\eqref{eq:Maxwell} such that
\begin{gather*}\label{eq:epsilonOgrade}
    \begin{gathered}
    \bm m^\varepsilon \in C([0,T^*];H^l(\mathbb R^3;\mathbb S^2))\quad \text{and} \quad \partial_t \bm m^\varepsilon \in  C ( [0,T^*] ; H^{l-4}(\mathbb R^3;\mathbb R^3))\, , \\
    f^\varepsilon \in C([0,T^*];H^{l-2} (\mathbb R^3\times \mathbb R^3) ) \cap C^1([0,T^*];H^{l-3}(\mathbb R^3\times \mathbb R^3)), \\
    \bm E^\varepsilon, \; \bm B^\varepsilon \in C([0,T^*];H^{l-2}(\mathbb R^3;\mathbb R^3)) \cap C^1 ( [0,T^*] ; H^{l-3}(\mathbb R^3;\mathbb R^3)) \, ,
    \end{gathered}
\end{gather*}
and $\mathrm{supp} \, f^\varepsilon(t) \subseteq B_{2R} \times B_{2R}$ for all $t\in[0,T^*]$.
\end{prop}
\begin{proof}
We set up an iteration scheme: Starting from $\bm j_0^{\varepsilon}= - \int_{\mathbb{R}^3} \bm v f_0^{\varepsilon} \, d\bm v$, there exists,
by Theorem \ref{thm:LLGglobalno}, a global (unique) solution
$\bm m_1^{\varepsilon}$ of \eqref{eq:LLG1}. The solution gives rise to emergent fields $\bm e_1^{\varepsilon}$ and $\bm b_1^{\varepsilon}$ according to \eqref{eq:emergent_fields}. Moreover, by virtue of Theorem \RN{1} from \cite{Kato1975}, there exist unique $\bm E_1^{\varepsilon}$ and $\bm B_1^{\varepsilon}$ solving \eqref{eq:Maxwell} for the given initial fields. After changing the equation \eqref{eq:Vlasov0} like in \cite{Wollman1984} to obtain integrable coefficients, Theorem \RN{1} from \cite{Kato1975} provides a unique global
solution $f_1^{\varepsilon}$ with the total Lorentz force $\bm F_0^{\varepsilon}$ for the given initial distribution. Since we can prove that the support of $f_1^{\varepsilon}$ remains bounded for finite time, it coincides with the solution to equation \eqref{eq:Vlasov0}, providing an update $\bm j_1^{\varepsilon}$. Hence, we arrive to the following iterating scheme with the LLG equation
\begin{align*}
    \begin{aligned}
    \partial_t \bm m^\varepsilon_n = \bm m^{\varepsilon}_n \times \left( \alpha \, \partial_t \bm m^\varepsilon_n + \Delta \bm m ^\varepsilon_n + \Delta^2 \bm m^\varepsilon_n - h\, \bm{\hat {e}}_3 \right)  - (K_\varepsilon \, \bm j^\varepsilon_{n-1} \cdot \nabla ) \bm m^\varepsilon_n \,,
    \end{aligned}
\end{align*}
the Vlasov equation
\begin{equation*}
    \partial_t f_n^\varepsilon + \bm v \cdot \nabla_x \, f_n^\varepsilon + (K_\varepsilon \bm F_{n-1}^{\varepsilon}) \cdot \nabla_v\, f_n^\varepsilon = 0 \, 
\end{equation*}
with Lorentz force $\bm F_{n-1}^{\varepsilon} =  - \left( \bm E^\varepsilon_{n-1} + \bm e^\varepsilon_{n-1} + \bm v \times ( \bm B^\varepsilon_{n-1} +  \bm b_{n-1}^\varepsilon) \right)$ and Maxwell equations
\begin{equation*}
    \varepsilon_r \, \partial_t \bm E_n^\varepsilon - \frac{1}{\mu_r} \, \nabla \times \bm B_{n}^\varepsilon = - K_\varepsilon \, \bm j_{n-1}^\varepsilon \, , \quad
    \partial_t \bm B_n^\varepsilon + \nabla \times \bm E_{n}^\varepsilon = 0 \, .
\end{equation*}
The smoothing properties of $ K_\varepsilon $, the compact support of $f_{n}^\varepsilon$, and the fact that $H^l(\mathbb R^3)$ is an algebra imply
\begin{gather*}
K_\varepsilon \, \bm j_{n-1}^\varepsilon , K_\varepsilon \, \bm E^\varepsilon_{n-1},  K_\varepsilon \, \bm e^\varepsilon_{n-1}, K_\varepsilon \, \bm B^\varepsilon_{n-1} , K_\varepsilon \, \bm b^\varepsilon_{n-1}\in C([0,\infty);H^{l-2}(\mathbb R^3 ; \mathbb R^3)) \, .
\end{gather*} 
Hence we find sequences such that
\begin{gather*}
    \bm m_{n}^\varepsilon  \in C([0,\infty); H^l(\mathbb R^3;\mathbb S^2)) \quad \text{and} \quad \partial_t \bm m_{n}^\varepsilon \in  C([0,\infty);H^{l-4}(\mathbb R^3;\mathbb R^3)) \, ,\\
    f^\varepsilon_{n} \in C([0,\infty); H^{l-2}(\mathbb R^3 \times \mathbb R^3))\cap C^1([0,\infty); H^{l-3}(\mathbb R^3 \times \mathbb R^3)) \, , \\
    \bm E^\varepsilon_{n}, \, \bm B^\varepsilon_{n} \in C([0,\infty);\, H^{l-2}(\mathbb R^3;\mathbb R^3))\cap C^1([0,\infty);\, H^{l-3}(\mathbb R^3;\mathbb R^3))\,.
\end{gather*}
It can be shown by arguments similar to \cite{Wollman1984} that there exists some terminal time $T^*>0$ small enough such that this sequence remains bounded in their respective function spaces and $\mathrm{supp} f^\varepsilon_n \subseteq B_{2R}\times B_{2R}$ uniformly in $n\,$. Since $l\geq 7$ is large enough, we get pointwise compactness in space-time from Ascoli's Theorem. The limit functions
\begin{gather*}
    \begin{gathered}
    \bm m^\varepsilon  \in L^\infty((0,T^*);H^l(\mathbb R^3;\mathbb S^2)) \cap \dot W^{1,\infty}( (0,T^*) ; H^{l-4}(\mathbb R^3;\mathbb R^3))\, ,\\
    f^\varepsilon \in L^\infty((0,T^*); H^{l-2}(\mathbb R^3\times\mathbb R^3))\cap W^{1,\infty} ( (0,T^*) ; H^{l-3}(\mathbb R^3 \times \mathbb R^3)) \, , \\
    \bm E^\varepsilon, \; \bm B^\varepsilon \in L^\infty((0,T^*);H^{l-2}(\mathbb R^3;\mathbb R^3)) \cap W^{1,\infty} ( (0,T^*) ; H^{l-3}(\mathbb R^3;\mathbb R^3)) \, .
    \end{gathered}
\end{gather*}
solve the regularized system \eqref{eq:LLG1}-\eqref{eq:Maxwell}. In particular we have that
\begin{gather*}
 \bm m^\varepsilon  \in C([0,T^*];H^{l-4}(\mathbb R^3;\mathbb S^2)) \, , \\
    f^\varepsilon \in C([0,T^*];H^{l-3}(\mathbb R^3\times \mathbb R^3)) \, , \\
    \bm E^\varepsilon,\bm B^\varepsilon \in C([0,T^*];  H^{l-3} (\mathbb R^3, \mathbb R^3))\, .
\end{gather*}
Using the compactness of the support of $f^\varepsilon$ we get $\bm j^\varepsilon \in C([0,T^*];H^{l-3}(\mathbb R^3))$. By making use of the mollifier $K_\varepsilon$ again  we obtain from Theorem \ref{thm:LLGglobalno} the required regularity for $\bm m^\varepsilon$. Similarly using Theorem \RN{1} from \cite{Kato1975} we get that %
$f^\varepsilon, \bm E^\varepsilon$ and $\bm B^\varepsilon$ belong to spaces stated in the Proposition.
\end{proof}
\subsection*{Global solution to the regularized system}
Once we have a short time solution to the regularized system \eqref{eq:LLG1}-\eqref{eq:Maxwell} we can use the energy argument to extend the solution to a global one.

\begin{lem}
Let $\varepsilon>0$ be fixed. Then for the regularized system of equations \eqref{eq:LLG1}-\eqref{eq:Maxwell} we have the following energy-dissipation law
\begin{align}\label{eq:energyestimate2}
    \alpha \int_0^T  \|\partial_t \bm m^\varepsilon\|^2_{L^2} \, dt + \Big[ \mathbb{E}(f^\varepsilon ,\bm E^\varepsilon, \bm B^\varepsilon, \bm m^\varepsilon)\Big]_{t=0}^T = 0 \, ,
\end{align}
where the total energy is defined in \eqref{eq:abstract_energy}. 
\end{lem}
\begin{proof}
We note that the obtained local solution from the previous chapter is smooth and all of the calculus below is therefore rigorous. We multiply the Vlasov equation \eqref{eq:Vlasov0} by $|\bm v|^2$ and integrate by parts to get
\begin{equation}\label{eq:Vlasov}
   \frac{1}{2}\, \frac{d}{dt}  \int \int_{\mathbb R^3 \times \mathbb R^3} f^\varepsilon \, |\bm v|^2 \, d\bm v \, d\bm x = \langle \bm j^\varepsilon , \, K_\varepsilon \, \bm e^\varepsilon + K_\varepsilon \, \bm E^\varepsilon \rangle \, .
\end{equation}
Multiplying the Maxwell's equations \eqref{eq:Maxwell} by $\bm E^\varepsilon$ and $\bm B^\varepsilon /\mu_r$ respectively and integrating we get
\begin{equation}\label{eq:Maxwell1}
   \frac{1}{2}\, \frac{d}{dt} \int_{\mathbb R^3} \varepsilon_r \, |\bm E^\varepsilon|^2 + \frac{1}{\mu_r} \, |\bm B^\varepsilon|^2 \, d\bm x  = -\langle K_\varepsilon \, \bm j^\varepsilon , \, \bm E^\varepsilon \rangle \, .
\end{equation}
We then use $\bm m^\varepsilon \times \partial_t \bm m^\varepsilon $ as a test function for \eqref{eq:LLG1} and integrate by parts to get
\begin{equation}\label{eq:LL}
    \frac{1}{2}\,\frac{d}{dt} E(\bm m^\varepsilon) + \alpha \int_{\mathbb R^3} |\partial_t \bm m^\varepsilon |^2 \, d\bm x = - \langle K_\varepsilon \, \bm j^\varepsilon,\,  \bm e^\varepsilon \rangle \, .
\end{equation}
Since $K_\varepsilon$ is self-adjoint, adding \eqref{eq:Vlasov}, \eqref{eq:Maxwell1} and \eqref{eq:LL} together and integrating in time we obtain the energy estimate \eqref{eq:energyestimate2}.
\end{proof}

We would now like to show that our local solution does not explode at any arbitrary $T>0$ in order to extend the solution from $[0,T^*]$ to $[0,\infty)$. In particular, \eqref{eq:energyestimate2} yields an estimate for
\begin{equation*}
    \int \int_{\mathbb R^3 \times \mathbb R^3} f^{\varepsilon}(t) |\bm v|^2 \, d\bm v \, d \bm x \leq C  \quad \text{for all} \quad t\in [0,T)\, .
\end{equation*}
We recall that from the method of characteristics, we have $f^\varepsilon\geq 0$ and
\begin{equation}\label{eq:LpConservation}
    \|f^\varepsilon (t) \|_{L^p} = \| f_0^\varepsilon \|_{L^p} \quad \text{for all} \quad t \in [0,T), \; p\in[1,\infty] \, .
\end{equation}
By a simple Holder's inequality we then get $\|\bm j^\varepsilon (t) \|_{L^1}\leq C$ for all $t\in [0,T)$. We therefore know that $K_\varepsilon\, \bm j^\varepsilon$ remains bounded in $C([0,T);H^{l-2}(\mathbb R^3;\mathbb R^3))$. Then from \eqref{eq:Gronwallograda1} $\bm m^\varepsilon$ remains bounded in $C([0,T);H^l(\mathbb R^3;\mathbb S^2)$, more precisely
\begin{equation*}
    \|\bm m^\varepsilon(t) - \bm{\hat {e}}_3\|_{H^l} \leq C(T,\varepsilon) \quad \text{for all} \quad t\in [0,T)\, .
\end{equation*}
We use Theorem \RN{1} from \cite{Kato1975} to get that $\bm E^\varepsilon, \, \bm B^\varepsilon$ remain bounded in $C([0,T);H^{l-2}(\mathbb R^3 \times \mathbb R^3))$. From the characteristic equations we then get that $ \mathrm{supp} f^\varepsilon$ remains bounded for finite time $T>0$. We can now use Theorem \RN{1} from \cite{Kato1975} again to obtain that $f^\varepsilon$ is bounded in $C([0,T);H^{l-2}(\mathbb R^3 \times \mathbb R^3).$ Our solution can therefore be extended by continuity to a global solution to the regularized LLG-VM system \eqref{eq:LLG1}-\eqref{eq:Maxwell}.

\subsection*{Compactness} 
In this section, we finish the proof of Theorem \ref{thm:glavni} by passing to the limit $\varepsilon\rightarrow 0^+$. By this we mean we consider some sequence $\varepsilon_k\rightarrow 0^+$ and its subsequences when necessary without relabeling for simplicity. 

From the energy-dissipation law \eqref{eq:energyestimate2} and \eqref{eq:LpConservation}, the solutions to the regularized system \eqref{eq:LLG1}-\eqref{eq:Maxwell} given in the previous section are bounded, uniformly in $\varepsilon$, in their respective function spaces
\begin{gather}
    \bm m^\varepsilon \in C([0,\infty);H^2(\mathbb R^3;\mathbb R^3)) \cap \dot H^1((0,\infty);L^2(\mathbb R^3;\mathbb R^3)) \, , \label{eq:mUniformbound} \\
    f^\varepsilon \in C([0,\infty);L^1\cap L^r(\mathbb R^3\times \mathbb R^3) ) \, , \quad |\bm v|^2 f^\varepsilon \in C([0,\infty);L^1(\mathbb R^3\times \mathbb R^3)) \, ,  \label{eq:fUniformbound}\\
    \bm E^\varepsilon, \; \bm B^\varepsilon \in C([0,\infty);L^2(\mathbb R^3;\mathbb R^3)) \, . \label{eq:EBUniformbound}
\end{gather}
Together with Sobolev embedding, we obtain the following uniform bounds for emergent electromagnetic fields given by \eqref{eq:regularizedemergent}
\begin{align}
\begin{aligned}\label{eq:ebUniformbound}
  \bm e^\varepsilon &\in L^2((0,\infty)\, ; L^{3/2}(\mathbb R^3 ;\mathbb R^3)) \, , \\
\bm b^\varepsilon &\in C ([0,\infty) \, ; L^3(\mathbb R^3;\mathbb R^3)) \, .
\end{aligned}
\end{align}
Let 
\begin{equation*}
    \ell \defeq \frac{5\,r-3}{4\,r-2} > \frac{6}{5} \, .
\end{equation*}
Then from Lemma \ref{lem:jregularity} we have
\begin{align*}
    \| \bm j^\varepsilon (t) \|_{L^\ell} \leq C\, \| f^\varepsilon(t) \|_{L^r}^{r/(5r-3)} \left (\int \int _{\mathbb R^3 \times \mathbb R^3} f^\varepsilon(t) |\bm v|^2 \, d\bm v \, d\bm x \right) ^{(4r-3)/(5r-3)} \, ,
\end{align*} 
and thus we get the uniform bound for
\begin{equation}\label{eq:jUnifomBound}
    \bm j^\varepsilon \in C([0,\infty)\, ; L^1 \cap L^{\ell}(\mathbb R^3;\mathbb R^3)) \, .
\end{equation}
We can now prove existence of a solution for the LLG equation.
\begin{lem}\label{lem:LLGcompactness}
Let $(\bm m^\varepsilon, \bm j^\varepsilon)$ be the sequence solving the regularized LLG equations \eqref{eq:LLG1} satisfying uniform bounds \eqref{eq:mUniformbound} and \eqref{eq:jUnifomBound}. Then, up to a subsequence, there exist limits
\begin{gather}\label{eq:jregularity}
    \begin{gathered}
    \bm m \in L^\infty((0,\infty);H^2(\mathbb R^3 ; \mathbb S^2)) \cap \dot H^1((0,\infty);L^2(\mathbb R^3;\mathbb R^3)) \, , \\
    \bm j \in L^\infty((0,\infty);L^\ell (\mathbb R^3;\mathbb R^3)) \, ,
    \end{gathered}
\end{gather}
solving \eqref{eq:LLG} in the sense of distributions. In addition, for emergent electromagnetic fields given by \eqref{eq:regularizedemergent}, we have that
\begin{alignat*}{3}
    e^\varepsilon_i &\xrightarrow[\varepsilon]{} e_i = \bm m \cdot (\partial_i \bm m \times \partial_t \bm m) \quad &&\text{in} \quad \mathscr D'((0,\infty) \times \mathbb R^3) \, , \\
    b^\varepsilon_i &\xrightarrow[\varepsilon]{} b_i = \epsilon^{ijk} \, \bm m \cdot (\partial_j \bm m \times \partial_k \bm m) \quad &&\text{in} \quad \mathscr D'((0,\infty) \times \mathbb R^3) \, .
\end{alignat*}
\end{lem}
\begin{proof}[\textbf{Proof of Lemma \ref{lem:LLGcompactness}}]
 By \eqref{eq:mUniformbound} the functions $\bm m^\varepsilon-\bm{\hat {e}}_3$ are weakly* compact in the energy space
 \begin{equation*}
     \mathcal E= L^\infty((0,\infty)\, ; H^2(\mathbb R^3;\mathbb R^3)) \cap \dot H^1((0,\infty)\,;L^2(\mathbb R^3;\mathbb R^3))  \, .
 \end{equation*}
 Up to a subsequence, by Aubin-Lions lemma we may assume that, for some $\bm m$ such that $\bm m - \bm{\hat {e}}_3 \in \mathcal E$
 \begin{alignat}{2}
     \bm m^\varepsilon - \bm{\hat {e}}_3 &\xrightharpoonup[\varepsilon]{} \bm m - \bm{\hat {e}}_3 \quad &&\text{weakly* in } \mathcal E\, , \label{eq:weakcvg} \\
     \bm m ^\varepsilon&\xrightarrow[\varepsilon]{} \bm m \quad &&\text{in} \quad C ([0,T]\, ; H^s_{loc}( \mathbb R^3\,;\mathbb R^3)) \quad \text{for all} \quad s<2\, . \label{eq:strongcvg}
 \end{alignat}
From $|\bm m^\varepsilon|=1$ and \eqref{eq:strongcvg} we have $|\bm m |=1$ almost everywhere in $(0,\infty)\times \mathbb R^3$. Hence
 \begin{align*}
\bm m \in L^\infty((0,\infty)\, ;H^2(\mathbb R^3;\mathbb S^2))\, \cap\, \dot H^1((0,\infty)\,;L^2(\mathbb R^3;\mathbb R^3))\,.
\end{align*}
From \eqref{eq:strongcvg} and Sobolev embedding we have
 \begin{alignat}{3}\label{eq:Lqkonvergencija}
 \begin{aligned}
     \bm m^\varepsilon &\xrightarrow[\varepsilon]{} \bm m \quad &&\text{in} \quad L^p((0,T)\, ;L^p_{loc}(\mathbb R^3\,;\mathbb R^3))  \, ,\\
     \nabla \bm m^\varepsilon &\xrightarrow[\varepsilon]{} \nabla \bm m \quad &&\text{in} \quad L^p((0,T)\, ;L^q_{loc}(\mathbb R^3\,;\mathbb R^3))  \, ,
     \end{aligned}
 \end{alignat}
where $1\leq p \leq \infty$ and $1\leq q < 6$.
From \eqref{eq:weakcvg} we deduce that in particular
 \begin{align*}
     \partial_t \bm m^\varepsilon \xrightharpoonup[\varepsilon]{} \partial_t \bm m\, \quad \text{and} \quad \Delta \bm m^\varepsilon \xrightharpoonup[\varepsilon]{} \Delta \bm m \quad \text{weakly in} \quad L^2((0,\infty)\times \mathbb R^3 ;\mathbb R^3).
 \end{align*}
Therefore using \eqref{eq:Lqkonvergencija} we get that
\begin{alignat*}{2}
    \bm m^\varepsilon \times \partial_t \bm m^\varepsilon &\xrightarrow[\varepsilon]{} \bm m \times \partial_t \bm m \quad &&\text{in} \quad \mathscr D'((0,\infty)\times \mathbb R^3) \, ,\\
    \bm m^\varepsilon \times \Delta \bm m^\varepsilon &\xrightarrow[\varepsilon]{} \bm m \times \Delta \bm m \quad &&\text{in} \quad \mathscr D'((0,\infty)\times \mathbb R^3) \, ,\\
    \Delta (\bm m^\varepsilon \times \Delta \bm m^\varepsilon ) &\xrightarrow[\varepsilon]{} \Delta (\bm m \times \Delta \bm m ) \quad &&\text{in} \quad \mathscr D'((0,\infty)\times \mathbb R^3) \, ,\\
     \partial_k \left(\partial_k \bm m^\varepsilon \times  \Delta \bm m^\varepsilon \right) &\xrightarrow[\varepsilon]{} \partial_k \left(\partial_k \bm m \times \Delta \bm m \right) \quad &&\text{in} \quad \mathscr D'((0,\infty)\times \mathbb R^3) \, .
\end{alignat*}
From \eqref{eq:jUnifomBound} there exists $\bm j \in L^\infty((0,\infty);L^{\ell} (\mathbb R^3;\mathbb R^3))$ such that, up to a subsequence
\begin{align*}
    \bm j^\varepsilon \xrightharpoonup[\varepsilon]{} \bm j \quad \text{weakly* in} \quad L^\infty((0,\infty);L^{\ell} (\mathbb R^3;\mathbb R^3)) \, . 
\end{align*}
Since $\ell^* < (6/5)^*=6\,$, using \eqref{eq:Lqkonvergencija} we get that
\begin{equation*}
    (K_\varepsilon\,\bm j^\varepsilon \cdot \nabla)\bm m^\varepsilon \xrightarrow[\varepsilon]{}  (\bm j \cdot \nabla) \bm m \quad \text{in} \quad \mathscr D'((0,\infty)\times \mathbb R^3)\, .
\end{equation*}
In view of \eqref{eq:divergencestructure} the above convergences prove that $\bm m$ is a weak solution to the LLG equation \eqref{eq:LLG}. Moreover from \eqref{eq:Lqkonvergencija} we have that $\bm e^\varepsilon,\, \bm b^\varepsilon$ weakly converge to $\bm e,\, \bm b,$ respectively, i.e.
\begin{alignat*}{3}
\bm m^\varepsilon \cdot (\nabla \bm m^\varepsilon \times \partial_t \bm m^\varepsilon ) &\xrightarrow[\varepsilon]{} \bm m \cdot (\nabla \bm m \times \partial_t \bm m) \quad &&\text{in} \quad \mathscr D'((0,\infty) \times \mathbb R^3) \, , \\
\epsilon^{ijk}\, \bm m^\varepsilon \cdot (\partial_j \bm m^\varepsilon \times \partial_k \bm m^\varepsilon) &\xrightarrow[\varepsilon]{} \epsilon^{ijk} \, \bm m \cdot (\partial_j \bm m \times \partial_k \bm m) \quad &&\text{in} \quad \mathscr D'((0,\infty) \times \mathbb R^3) \, . \quad \mbox{\qedhere}
\end{alignat*}
\end{proof}

It remains to show compactness for terms in the Vlasov equation \eqref{eq:Vlasov0}. The key ingredient is the following velocity averaging lemma. 
\begin{lem}\label{lem:Stability}
Let $(f^\varepsilon,\bm E^\varepsilon, \bm B^\varepsilon, \bm e^\varepsilon, \bm b^\varepsilon)$ be the sequence solving the regularized Vlasov equation \eqref{eq:Vlasov0} satisfying uniform bounds \eqref{eq:fUniformbound}-\eqref{eq:ebUniformbound} and $\psi \in \mathscr{D} (\mathbb R^3) $. Then, up to a subsequence, we have 
\begin{equation}\label{eq:L1compactness}
     \int_{\mathbb R^3} \, f^\varepsilon \psi \, d\bm v \xrightarrow[\varepsilon]{} \int_{\mathbb R^3}  \, f \psi \, d\bm v \quad \text{in} \quad L^1_{loc}((0,\infty)\times \mathbb R^3) \, .
\end{equation}
\end{lem}
\begin{proof}[\textbf{Proof of Lemma \ref{lem:Stability}}]
We will prove the Lemma in two steps like in \cite{DiPernaLions1989}. In the first step, we additionally assume that $f^\varepsilon$ is uniformly bounded in $L^\infty((0,\infty)\times \mathbb R^3 \times \mathbb R^3)$. In the second step we prove for general $f^\varepsilon$.

\textbf{\textit{Step 1}} We first remark that we can rewrite \eqref{eq:Vlasov0} as
\begin{equation*}
    \partial_t f^\varepsilon + \bm v \cdot \nabla_x f^\varepsilon = \nabla_v \cdot \bm g_2^\varepsilon \, ,
\end{equation*}
where
\begin{equation*}
    \bm g_2^\varepsilon=  K_\varepsilon \, (\bm E^\varepsilon + \bm e^\varepsilon + \bm v \times (\bm B^\varepsilon + \bm b^\varepsilon )) \, f^\varepsilon \, .
\end{equation*}
Then if $f^\varepsilon$ is uniformly bounded in $L^\infty((0,T)\times \mathbb R^3 \times \mathbb R^3)$ (for all $T<\infty$), from \eqref{eq:energyestimate2} we deduce that $\bm g_2^\varepsilon$ is uniformly bounded in $L^{2}_{loc}(\mathbb R^3_v\,;L^{3/2}((0,T)\times \mathbb R^3_x))$ (for all $T<\infty$).

Next, for all $\delta>0$ and $T>0$, we choose $\zeta \in \mathscr D(\mathbb R)$ such that $\zeta \equiv 1 $ on $[\delta, T]$, $\mathrm{supp} \, \zeta \subset [\frac{1}{2} \, \delta, 2T]$ and we observe that $\tilde f^\varepsilon \defeq \zeta \, f^\varepsilon$ solves
\begin{align*}
    \partial_t \tilde f^\varepsilon + \bm v \cdot \nabla_x \tilde f^\varepsilon = {g}_1^\varepsilon + \nabla_v \cdot  \tilde{ \bm g_2}^\varepsilon
\end{align*}
where both ${g}_1^\varepsilon = \zeta ' f^\varepsilon$ and $\tilde{ \bm g_2}^\varepsilon=\zeta \, \bm g_2^\varepsilon$ are uniformly bounded in ${L^2_{loc}(\mathbb R^3_v \, ; L^{3/2}(\mathbb R \times \mathbb R^3_x))}$. Remark also that $\tilde f^\varepsilon$ is uniformly bounded in ${L^2(\mathbb R \times \mathbb R^3 \times \mathbb R^3)}$ and that
\begin{equation*}
3/2 > 2(N+1)/(N+3)=4/3\, .
\end{equation*}
Therefore if $\psi \in \mathscr D (\mathbb R^3)$, we deduce from Theorem \ref{thm:dipernalions} that
\begin{equation*}
\int_{\mathbb R^3} \tilde f^\varepsilon \, \psi \, d\bm v \quad \text{is bounded in} \quad H^{1/12}(\mathbb R \times \mathbb R^3) \, .
\end{equation*}
In particular, we get
\begin{equation*}
    \int_{\mathbb R^3}\zeta \, f^\varepsilon \psi \, d\bm v \xrightarrow[\varepsilon]{} \int_{\mathbb R^3} \zeta \, f \psi \, d\bm v \quad \text{in} \quad L^1_{loc}((0,T)\times \mathbb R^3) \, .
\end{equation*}
Since $\zeta \equiv 1$ on $[\delta,T]$, we conclude that 
\begin{equation*}
     \int_{\mathbb R^3} \, f^\varepsilon \psi \, d\bm v \xrightarrow[\varepsilon]{} \int_{\mathbb R^3}  \, f \psi \, d\bm v \quad \text{in} \quad L^1_{loc}((0,\infty)\times \mathbb R^3) \, ,
\end{equation*}
provided we show that 
\begin{equation*}
    \sup_{\varepsilon} \int_0^\delta \int \int_{B_R\times B_R} |f^\varepsilon| \, d\bm x \, d\bm v \, dt \rightarrow 0 \quad \text{as} \quad \delta \rightarrow 0_+ \, .
\end{equation*}
This claim is an immediate consequence of the $L^r((0,T)\times \mathbb R^3 \times \mathbb R^3)$ bound on $f^\varepsilon$.

\textbf{\textit{Step 2}} We then define a $C^1([0,\infty))$ function
\begin{align*}
    \beta_\delta(t) = \frac{t}{1+\delta t} \quad \text{for} \quad 0<\delta < 1,\; t\geq 0 \, ,
\end{align*}
and therefore $\beta_\delta (f^\varepsilon)$ solves
\begin{align*}
    \partial_t ( \beta_\delta(f^\varepsilon)) + \bm v \cdot \nabla_x\, \beta_\delta(f^\varepsilon) = \nabla_v \cdot \bm g_2^\varepsilon \, ,
\end{align*}
where 
\begin{equation*}
    \bm g_2^\varepsilon \defeq  K_\varepsilon \, ( \bm E^\varepsilon +  \bm e^\varepsilon + \bm v \times ( \bm B^\varepsilon +  \bm b^\varepsilon )) \,\beta_\delta( f^\varepsilon) \, .
\end{equation*}
Since $\beta_\delta(t)\leq t$ and $\beta_\delta(t)\leq 1/\delta$ we know that $\beta_\delta(f^\varepsilon)$ is bounded in $L^1 \cap L^\infty(\mathbb R^3\times \mathbb R^3)$. Then from \eqref{eq:energyestimate2} we deduce that $\bm g_2^\varepsilon$ is uniformly bounded in $L^{2}_{loc}(\mathbb R^3_v \,;L^{3/2}((0,T)\times \mathbb R^3_x))$ (for all $T<\infty$). We can therefore apply the proof from Step 1 to obtain
\begin{align*}
    \int_{\mathbb R^3} \beta_\delta (f^\varepsilon) \, \psi \, d\bm v \xrightarrow[\varepsilon]{} \int_{\mathbb R^3} f_\delta \, \psi \, d\bm v \quad \text{in} \quad L^1_{loc}((0,\infty)\times \mathbb R^3)\, ,
\end{align*}
where $f_\delta$ is the weak limit of $\beta_\delta(f^\varepsilon)$ and $\psi\in \mathscr{D}(\mathbb R^3)$ is arbitrary. In order to conclude that 
\begin{equation*}
     \int_{\mathbb R^3} \, f^\varepsilon \psi \, d\bm v \xrightarrow[\varepsilon]{} \int_{\mathbb R^3}  \, f \psi \, d\bm v \quad \text{in} \quad L^1_{loc}((0,\infty)\times \mathbb R^3)  \, ,
\end{equation*}
we have to prove that
\begin{align*}
    \sup_{\varepsilon} \int_0^T \int_{B_R \times B_R} |f^\varepsilon - \beta_\delta(f^\varepsilon)| \, d\bm x \, d\bm v \, dt \rightarrow 0 \quad \text{as} \quad \delta \rightarrow 0_+ \quad \text{for all}\quad R,T<\infty \, . 
\end{align*}
Indeed, we have
\begin{align*}
    0\leq f^\varepsilon - \beta_\delta(f^\varepsilon) = \frac{\delta \, (f^\varepsilon) ^2}{1+\delta f^\varepsilon} \leq M\delta f^\varepsilon + f^\varepsilon \mathds{1}_{\{ f^\varepsilon > M\}} \, ,
\end{align*}
where $f^\varepsilon \mathds{1}_{\{ f^\varepsilon > M\}}$ can be made arbitrarily small in $L^1((0,T)\times B_R \times B_R)$ uniformly in $\varepsilon$ by taking $M$ large.
\end{proof}
We are now ready to finish the proof of Theorem \ref{thm:glavni}.
\begin{proof}[\textbf{Proof of Theorem \ref{thm:glavni}}] Let $(\bm m^\varepsilon, f^\varepsilon,\bm E^\varepsilon, \bm B^\varepsilon)$ be the sequence solving the regularized system \eqref{eq:LLG1}-\eqref{eq:Maxwell}. From Lemma \ref{lem:LLGcompactness} there exist $(\bm m,\bm j)$ solving \eqref{eq:LLG} in the sense of distributions. In exactly the same way as in \cite{DiPernaLions1989}, from Lemma \ref{lem:Stability} we get that
\begin{align*}
\bm j^\varepsilon \xrightarrow[\varepsilon]{} - \int_{\mathbb R^3} \bm v \, f \, d\bm v  \quad \text{in} \quad L^1_{loc}((0,\infty)\times \mathbb R^3)  \, ,
\end{align*}
and therefore $\bm j=- \int_{\mathbb R^3}\bm v \, f \, d\bm v \, .$ Moreover, since $\int_{\mathbb R^3} f^\varepsilon \, \psi \, d\bm v$ is uniformly bounded in $L^r_{loc}((0,\infty)\times \mathbb R^3)$ we deduce from Lemma \ref{lem:Stability} that 
\begin{equation*}
    \int_{\mathbb R^3}f^\varepsilon \psi \, d\bm v \rightarrow \int_{\mathbb R^3} f \psi \, d\bm v \quad \text{in} \quad L^p_{loc}((0,\infty)\times \mathbb R^3) \quad \text{for all} \quad 1\leq p <r.
\end{equation*}
In view of \eqref{eq:Vlasovtransport}, we obtain weak convergence of all terms in the Vlasov equation \eqref{eq:Vlasov0} taking into account uniform bounds \eqref{eq:EBUniformbound}-\eqref{eq:ebUniformbound} and $r>3$.
Moreover, from weak convergence of $\bm E^\varepsilon,\,\bm B^\varepsilon$ we get 
\begin{equation*}
    \varepsilon_r \, \partial_t \bm E - \frac{1}{\mu_r} \, \nabla \times \bm B = -\bm j\,,\quad \partial_t \bm B + \nabla \times \bm E = 0 \quad \text{in} \quad \mathscr D'((0,\infty)\times \mathbb R^3)\, .
\end{equation*}
From Sobolev embedding we have that $\bm m \in C^{1/2}([0,\infty);L^2(\mathbb R^3;\mathbb S^2))$ and thus the initial data is attained by the limit, i.e. $\bm m|_{t=0}=\bm m_0$. Then by interpolation we have
 \begin{align*}
    \| \bm m (t_1) - \bm m (t_2)\|_{H^s} &\leq  \| \bm m (t_1) - \bm m (t_2)\|_{H^2}^{s/2}  \| \bm m (t_1) - \bm m (t_2)\|_{L^2}^{1-s/2} \\
     &\leq \| \bm m \|^{s/2}_{L^\infty_t H^2_x} \, |t_1 - t_2|^{1/2 - s/4}\, .
 \end{align*}
 In particular we have $\bm m \in C([0,\infty);H^s(\mathbb R^3;\mathbb S^2))$ for $s<2$.
The VM system \eqref{eq:Vlasov0}-\eqref{eq:Maxwell} combined with the uniform bounds \eqref{eq:fUniformbound}-\eqref{eq:ebUniformbound} gives that 
\begin{gather*}
    \partial_t f^\varepsilon \quad \text{is bounded in} \quad L^2((0,\infty); W^{-1,\,3r/(2r+3)}_{loc}(\mathbb R^3 \times \mathbb R^3)) \, , \\
    \partial_t \bm E^\varepsilon , \; \partial_t \bm B^\varepsilon \quad \text{are bounded in} \quad L^\infty((0,\infty); H^{-2}(\mathbb R^3))\, .
\end{gather*}
These bounds imply that $\bm E^\varepsilon, \bm B^\varepsilon$ are compact in $C([0,T];H^{-s}_{loc}(\mathbb R^3))$ and that $f^\varepsilon$ is compact in $C([0,T];W_{loc}^{-s,\,3r/(2r+3)} (\mathbb R^3 \times \mathbb R^3))$ (for all $T<\infty$, all $s>0$). This, in particular, shows that the initial data are attained by the limit, i.e. $f|_{t=0} = f_0,\ \bm E|_{t=0}=\bm E_0, \ \bm B|_{t=0}=\bm B_0$.

The statement $f\geq 0$ follows from $f_0\geq 0$ and Mazur's lemma. Lastly, the proof to obtain
\begin{equation*}
    \nabla \cdot \bm E = \frac{\rho}{\varepsilon_r} \,, \quad \nabla \cdot \bm B= 0  \quad \text{in}\quad \mathscr D'((0,\infty)\times \mathbb R^3) \, ,
\end{equation*}
and that the mass $\int \int_{\mathbb R^3 \times \mathbb R^3} f \, d\bm x \, d \bm v $ is independent of $t\geq 0$ is analogous to \cite{DiPernaLions1989}.
\end{proof}

Finally we show that helicity functional behaves continuously along the flow of $\bm m$.

\begin{lem}\label{lem:hopfion} The Hopf invariant $H=H(\bm m)$ is a smooth functional over the class of fields
$\bm m: \mathbb{R}^3 \to \mathbb{S}^2$ such that $\bm m  \in H^s(\mathbb{R}^3; \mathbb{S}^2)$ where $s>3/2$.
\end{lem}
\begin{proof}It is clear that $H(\bm m)$ is independent of the special choice of $\bm a$.
By translation invariance and Sobolev embedding
\begin{align*}
|\bm b| \lesssim  |\nabla \bm m|^2 \in L^1 \cap L^{\frac{3}{5-2s}}(\mathbb{R}^3).
\end{align*}
By the properties of the Biot-Savart operator $\bm b \mapsto \bm a$ given by the singular integral
\begin{align*}
\bm a(\bm x) = \int_{\mathbb{R}^3} \frac{\bm b(\bm y) \times (\bm x -\bm y)}{|\bm x - \bm y|^3} d \bm y
\end{align*}
we obtain $|\bm a| \in L^p(\mathbb{R}^3)$ for all 
$3/2< p < \frac{3}{2(2-s)}$ which is larger than the dual exponent of $\frac{3}{5-2s}$, and the claim follows.
\end{proof}

\section{Uniqueness for the LLG equation}\label{section:uniqueness}

In this section, we prove Theorem \ref{thm:uniqueness}, i.e., uniqueness for
weak solutions to the LLG equation \eqref{eq:LLG} for the fixed current
$\bm j \in L^\infty((0,\infty); L^{6/5}(\mathbb{R}^3;\mathbb{R}^3))$. In particular, since in \eqref{eq:jregularity} we have $\ell>6/5$, the result holds for the solution given by Theorem \ref{thm:glavni}.

 \begin{proof}[\textbf{Proof of Theorem \ref{thm:uniqueness}}]We start from a weak solution to the LLG equation \eqref{eq:LLG}, i.e. the solution $\bm m$ satisfies
\begin{align*}
    \langle \partial_t \bm m, \bm v \rangle &= \alpha \langle \bm m \times \partial_t \bm m, \bm v \rangle + \langle \bm m \times (\Delta \bm m - h\, \bm{\hat {e}}_3), \bm v\rangle + \langle\bm m \times \Delta \bm m, \Delta \bm v\rangle \\
    \ &\phantom{{} =  }+ 2 \sum_{k=1}^3 \langle \partial_k \bm m \times \Delta \bm m , \partial_k \bm v \rangle - \langle (\bm j \cdot \nabla ) \bm m, \bm v\rangle \, , 
\end{align*}
for almost every $t \in [0,\infty)$ and for all $\bm v \in H^2(\mathbb R^3;\mathbb R^3).$ Since $H^2(\mathbb R^3)$ is closed under pointwise multiplication $\bm v \times \bm m \in H^2(\mathbb R^3;\mathbb R^3)$ is a valid test function and using the identity $\langle \bm m \times \partial_t \bm m ,\bm v \rangle = \langle \partial_t \bm m , \bm v \times \bm m\rangle$ we can pass to the Landau-Lifshitz formulation
\begin{align}\label{eq:LLGpravioblikweakform}
    (1+\alpha^2) \langle \partial_t \bm m, \bm v\rangle +\langle A(\bm m )\Delta^2 \bm m , \bm v \rangle = \langle A(\bm m) \bm f , \bm v \rangle -\alpha \langle \Lambda \bm m , \bm v \rangle  \, ,
\end{align}
where we use the following convenient notations
\begin{gather*}
    \langle \Delta^2 \bm m , \bm v \rangle \defeq \langle \Delta \bm m, \Delta \bm v \rangle \, , \quad \langle \bm m \times \Delta^2 \bm m , \bm v \rangle  \defeq \langle \Delta \bm v \times \bm m , \Delta \bm m \rangle + 2\sum_{k=1}^3 \langle \partial_k \bm v \times \partial_k \bm m, \Delta \bm m \rangle \, , \\
    \langle \Lambda \bm m , \bm v \rangle = - \langle (\bm m \cdot \Delta^2 \bm m ) \bm m , \bm v \rangle \defeq - \big\langle |\Delta \bm m |^2 \bm m , \bm v \big\rangle  + 2 \,\big\langle D\bm m \otimes D \bm m , D^2(\bm v \cdot \bm m )\big\rangle \, .
\end{gather*}
where $\left \langle  D \bm m \otimes D \bm m , D^2(\bm v \cdot \bm m) \right \rangle $ stands for $\sum_{i,j} \left \langle \partial _i \bm m \cdot \partial_j \bm m , \partial_{ij}^2 (\bm v \cdot \bm m ) \right \rangle $. Details can be found in \cite{ChugreevaMelcher2017}. We now assume there exist two distinct solutions to the LLG equation \eqref{eq:LLG}, $\bm m_1$ and $\bm m_2$. We can then subtract the two equations of form \eqref{eq:LLGpravioblikweakform} and integrate in time. We choose the test function to be
\begin{equation*}
    \bm v (t) = \bm m_1 (t) - \bm m_2 (t) \, ,
\end{equation*}
and get
\begingroup
\allowdisplaybreaks
\begin{subequations}
\begin{align}
    \frac{1+\alpha^2}{2}\| \bm v (T)\|^2_{L^2} = &-\int_0^T \alpha \, \langle \Delta \bm v, \Delta \bm v \rangle \, dt \label{eq:subeqa} \\
    &+ \int_0^T \alpha \, \big\langle  |\Delta \bm m_1|^2 \bm m_1 - |\Delta \bm m_2|^2 \bm m_2, \bm v \big\rangle \, dt \\
    &- \int_0^T \alpha \, \big\langle D \bm m_1 \otimes D \bm m_1 , D^2 (\bm m_1 \cdot \bm v) \big \rangle \, dt \\
    &+\int_0^T \alpha \, \big\langle D \bm m_2 \otimes D \bm m_2 , D^2 (\bm m_2 \cdot \bm v) \big \rangle \, dt \\
    &+ \int_0^T \langle \Delta \bm m_1 , \bm m_1 \times \Delta \bm v \rangle \, dt - \int_0^T \langle \Delta \bm m_2, \bm m_2 \times \Delta \bm v \rangle \, dt \\
    &+ 2\, \sum_{k=1}^3 \left [ \int_0^T  \langle \Delta \bm m_1 ,  \partial_k \bm m_1 \times \partial_k \bm v \rangle \, dt - \int_0^T  \langle \Delta \bm m_2 , \partial_k \bm m_2 \times \partial_k \bm v \rangle \, dt \right ] \\
    &+ \int_0^T \alpha \, \langle \bm m_1 \times \bm m_1 \times \Delta \bm m_1 - \bm m_2 \times \bm m_2 \times \Delta \bm m_2 , \bm v \rangle \, dt \\
    &+\int_0^T \langle \bm m_1 \times \Delta \bm m_1 - \bm m_2 \times \Delta \bm m_2 , \bm v \rangle \, dt \, \label{eq:subeqh} \\
    &- \int_0^T \alpha \, \langle \bm m_1 \times \bm m_1 \times h \, \bm{\hat {e}}_3) - \bm m_2 \times \bm m_2 \times h \, \bm{\hat {e}}_3) , \bm v \rangle \, dt \label{eq:subeqi} \\
    &- \int_0^T \left \langle \bm m_1 \times h \, \bm{\hat {e}}_3) - \bm m_2 \times h \, \bm{\hat {e}}_3) ,\bm v \right \rangle \, dt \label{eq:subeqj}\\
    &- \int_0^T \alpha \, \langle \bm m_1 \times (\bm j \cdot \nabla) \bm m_1 - \bm m_2 \times (\bm j \cdot \nabla) \bm m_2 , \bm v \rangle  \, dt \label{eq:subeqk}\\
    &- \int_0^T \langle (\bm j \cdot \nabla) \bm v , \bm v \rangle \, dt. \label{eq:subeql}
\end{align}
\end{subequations}
\endgroup
The estimates concerning the terms \eqref{eq:subeqa}-\eqref{eq:subeqh} were obtained in \cite{ChugreevaMelcher2017}. We estimate \eqref{eq:subeqi} by
\begin{align*}
    \langle \bm m_1 \times \bm m_1 \times h \, \bm{\hat {e}}_3) - \bm m_2 \times \bm m_2 \times h \, \bm{\hat {e}}_3) , \bm v \rangle &= h\, \big\langle (\bm m_1 \cdot \bm{\hat {e}}_3) \bm v + (\bm v \cdot \bm{\hat {e}}_3) \bm m_2, \bm v\big\rangle \leq C \, \|\bm v\|_{L^2}^2 \, .
\end{align*}
Estimate for \eqref{eq:subeqj} goes similarly. It remains to treat the terms \eqref{eq:subeqk} and \eqref{eq:subeql}.
\begin{align*}
    \langle \bm m_1 \times (\bm j \cdot \nabla) \bm m_1 - \bm m_2 \times (\bm j \cdot \nabla) \bm m_2 , \bm v\rangle &= \langle \bm m_2 \times (\bm j \cdot \nabla) \bm v, \bm v\rangle \\
    & \leq  \| \bm j \|_{L^{6/5}} \| \nabla \bm v \|_{L^6} \| \bm v \|_{L^{\infty}} \\
    &\leq \| \bm j \|_{L^{6/5}} \| \bm v \|_{H^2}^{7/4} \| \bm v \|_{L^2}^{1/4} \, .
\end{align*}
%
The last term \eqref{eq:subeql} is estimated in the same way. We therefore get
\begin{align*}
    \frac{1}{2} \| \bm v(T) \|_{L^2}^2 + \lambda \int_0^T \| \Delta \bm v \|_{L^2}^2 \, dt \leq C \int_0^T R \big(\|\bm v \|_{L^2}, \|\bm v\|_{H^2} \big) \, dt \, ,
\end{align*}
where the function $R(a,b)$ is given by
\begin{align*}
    R(a,b)= a^2 + ab+a^{1/4} b^{7/4}+ a^{1/2} b^{3/2} \, .
\end{align*}
Using Young's inequality and absorbing the $\|\Delta \bm v \|_{L^2}$ terms on the left-hand side we obtain
\begin{align*}
     \| \bm v(T) \|_{L^2}^2 + \lambda \int_0^T \| \Delta \bm v \|_{L^2}^2 \, dt \leq C \int_0^T \|\bm v \|_{L^2}^2 \, dt \, .
\end{align*}
We conclude with the Gronwall's lemma that $\|\bm v(T)\|^2_{L^2}=0$ for all $T\in [0,\infty)$.
\end{proof}
\paragraph{Acknowledgments.} We would like to thank Martin Frank for numerous helpful and stimulating discussions on transport equations. 
This work was funded by the Deutsche Forschungsgemeinschaft (DFG)
RTG 2326 \textit{Energy, Entropy, and Dissipative Dynamics}.
\bibliographystyle{siam}
\bibliography{references}

\end{document}